\documentclass[reqno]{amsart}
\usepackage{amssymb,upgreek,colordvi}
\usepackage{pdfsync}
\usepackage[usenames]{color}
\usepackage{color}
\newcommand{\highlight}[1]{\colorbox{white}{#1}}
\numberwithin{equation}{section}
\newcommand{\nn}{\nonumber}
   
\newtheorem{theorem}{Theorem}[section]
\newtheorem{proposition}[theorem]{Proposition}

\theoremstyle{definition}

\theoremstyle{definition} %%{remark}
\newtheorem{remark}[theorem]{Remark}

\newcommand{\subsubset}{\subset\joinrel\subset}

\newcommand{\bea}{\begin{eqnarray}}
\newcommand{\eea}{\end{eqnarray}}
\newcommand{\beas}{\begin{eqnarray*}}
\newcommand{\eeas}{\end{eqnarray*}}
\newcommand{\beq}{\begin{equation}}
\newcommand{\eeq}{\end{equation}}

\newcommand{\cA}{\mathcal A}

\newcommand{\cD}{\mathcal D}
\newcommand{\cE}{\mathcal E}

\newcommand{\cS}{\mathcal S}

\newcommand{\cMH}{\mathcal{MH}}
\newcommand{\cP}{\mathcal P}
\newcommand{\cSM}{\mathcal{SM}}
\newcommand{\cT}{\mathcal T}

\newcommand{\C}{\mathbb C}
\newcommand{\R}{\mathbb R}

\newcommand{\LL}{\mathbb L}
\newcommand{\EE}{\mathbb E}
\newcommand{\FF}{\mathbb F}
\newcommand{\HH}{\mathbb H}
\newcommand{\PP}{\mathbb P}

\newcommand{\cSX} {{\mathcal S}{\mathcal X}}

\newcommand{\ve}{\varepsilon}
\DeclareMathSymbol{\complement}{\mathord}{AMSa}{"7B}

\def\vv<#1>{\langle #1\rangle}
\def\Vv<#1>{\bigl\langle #1\bigr\rangle}

\begin{document}

% TOPMATTER

\title[Stability Of Equilibrium Shapes]
{Stability Of Equilibrium Shapes In Some Free Boundary Problems Involving Fluids}

\author[G.~Simonett]{Gieri Simonett}
\address{Department of Mathematics\\
         Vanderbilt University \\
         Nashville, TN~37240, USA}
\email{gieri.simonett@vanderbilt.edu}

\author[M. Wilke]{Mathias Wilke}
\address{Universit\"at Regensburg\\
         Fakult\"at f\"ur Mathematik \\
         D-93040, Germany}
\email{mathias.wilke@mathematik.uni-regensburg.de}

\thanks{The research of G.S.\ was partially
supported by the NSF Grant DMS-1265579.}

\begin{abstract}
In this chapter the motion of two-phase, incompressible, viscous fluids with surface tension is investigated.
Three cases are considered: (1) the case of heat-conducting fluids,
(2) the case of isothermal fluids, and (3) the case of Stokes flows.
In all three situations, the equilibrium states in the absence of outer forces are characterized
and their stability properties are analyzed. 
It is shown that the equilibrium states correspond to the critical points of a natural 
physical or geometric functional (entropy, available energy, surface area)
constrained by the pertinent conserved quantities (total energy, phase volumes).
Moreover, it is shown that solutions which do not develop singularities exist globally and 
converge to an equilibrium state.
\end{abstract}
%%%%%%%%%%%%%%%%%%%
\maketitle

%\vspace{-0.4cm}
%{\small\noindent
%{\bf Mathematics Subject Classification (2000):}\\
%{Primary: xx, xx; Secondary: xx.}
%\vspace{0.1in}\\
%{\bf Key words:} xx.
%\vspace{-0.1cm}
%%%%%%%%%%%%%%%%%%%%%%%%%%%%%%%%%%%%%%%%%%%%%%%%%%%%%%%%%%%%%%%%%%%%%%%%%%%%%%%%
%\begin{document}

%%%%%%%%%%%%%%%%%%%%%%%%%%%%%%%%%%%%%%%%%%%%%
\section{Introduction}\label{sect-intro}
%%%%%%%%%%%%%%%%%%%%%%%%%%%%%%%%%%%%%%%%%%%%%
In this chapter, the motion of two heat-conducting, incompressible, viscous 
Newtonian fluids in $\R^n$, $n\ge 2$, that are separated by a free interface is considered.
The position of the separating interface is unknown and has to be determined as part of the problem.

More precisely, the fluids are assumed to fill a bounded region $\Omega\subset\R^n$.
Let $\Gamma_0\subset\Omega$ be a given surface which bounds
the region $\Omega_{1}(0)$ occupied by an incompressible viscous 
fluid, $fluid_{1}$, called the {\em dispersed phase},
and let $\Omega_{2}(0)$ be the complement of
the closure of $\Omega_{1}(0)$ in $\Omega$, corresponding
to the region occupied by a second incompressible viscous fluid,
$fluid_{2}$, called the {\em continuous phase}. 
The two fluids are assumed to be immiscible, 
and the dispersed phase is assumed to not being in contact with the boundary 
$\partial\Omega$ of $\Omega$.

Let $\Gamma(t)$ denote the position of $\Gamma_{0}$
at time $t$. Thus, $\Gamma(t)$ is a sharp interface
which separates the fluids occupying the
regions $\Omega_1(t)$ and $\Omega_2(t)$, respectively.
The unit normal field on $\Gamma(t)$,
pointing from $\Omega_1(t)$ into $\Omega_2(t)$, is denoted by $\nu_\Gamma(t,\cdot)$.
Moreover, 
$V_\Gamma(t,\cdot)$ is the normal velocity and $H_\Gamma(t,\cdot):=-{\rm div}_\Gamma \nu_\Gamma$ 
the $(n-1)$-fold mean curvature (that is, the sum of the principal curvatures) of $\Gamma(t)$ 
with respect to $\nu_\Gamma(t,\cdot)$,
respectively. Here, 
$H_\Gamma$ is negative when $\Omega_1(t)$ is convex in a neighborhood of $x\in\Gamma(t)$. With this convention,
$H_\Gamma= -(n-1)/R$ if $\Gamma$ is a sphere of radius $R$ in $\R^n$.

The following notation for the physical quantities involved will be employed throughout this chapter:
\medskip\\
\begin{tabular}{ll}
 $\varrho_i$       & \hskip .2cm  the density of fluid$_i$  \\[.3em]
 $u_i$             & \hskip .2cm  the velocity of fluid$_i$ \\[.3em]
 $\pi_i$           & \hskip .2cm  the pressure of fluid$_i$ \\[.3em]
 $\theta_i$        & \hskip .2cm  the (absolute) temperature of fluid$_i$ \\[.3em]
 $\upmu_i(\theta_i)$   & \hskip .2cm  the shear viscosity of fluid$_i$ \\[.3em]
 $d_i(\theta_i)$     & \hskip .2cm  the heat conductivity of fluid$_i$ \\[.3em]
 $\sigma $           & \hskip .2cm  the surface tension \\[.3em]
 $D_i=D(u_i)=\frac{1}{2}([\nabla u_i] +[\nabla u_i]^{\sf T})$    & \hskip .2cm the rate of strain tensor  \\[.3em]
 $T_i=T(\theta_i,u_i,\pi_i)=2\upmu_i(\theta_i)D(u_i)-\pi_i I $  & \hskip .2cm the stress tensor.
 \end{tabular}

\bigskip
The densities $\varrho_i>0$ as well as the surface tension $\sigma>0$ are taken to be constant.

Let $\psi_i(\theta) $ denote the \highlight{\em Helmholtz free energy} of fluid$_i$.
Is is noted that the free energy is a constitutive quantity that depends on the physical properties
of the respective fluids.
Several physical and thermodynamic quantities are derived from $\psi_i$
as follows:

\bigskip

\begin{tabular}{ll}
$\epsilon_i(\theta)= \psi_i(\theta)+\theta\eta_i(\theta)$ 
	&\hskip .2cm  the (mass specific) \highlight{internal energy}  of fluid$_i$ \\[.3em]
$\eta_i(\theta) =-\psi_i^\prime(\theta)$
	 &\hskip .2cm the (mass specific) \highlight{entropy} of fluid$_i$ \\[.3em]
$\kappa_i(\theta)= e^\prime_i(\theta)=-\theta\psi_i^{\prime\prime}(\theta)$  
	&\hskip .2cm  the heat capacity of fluid$_i$. \\[.3em]
 \end{tabular}

\bigskip
\noindent
In the sequel, the index $i$ is occasionally dropped, but it should be noted that
the quantities $\psi$, $\epsilon$, $\eta$, $\kappa$, $\upmu$, and $d$
may have a jump across the interface $\Gamma$.

It is assumed throughout that the velocity $u$ and the temperature $\theta$ be continuous across $\Gamma$.
The motion of two heat-conducting, incompressible, viscous 
Newtonian fluids is then described by the following coupled system
\begin{equation}
\label{NS-heat}
\begin{aligned}
 \varrho\big(\partial_tu+(u|\nabla)u\big)- {\rm div}\, T
                         &=  0 &&\hbox{in} &&\Omega\setminus\Gamma(t),\\
{\rm div}\,u & = 0 &&\hbox{in} &&\Omega\setminus\Gamma(t),\\
 u &=0      &&\hbox{on} &&\partial\Omega,\\    
{[\![u]\!]} & = 0  &&\hbox{on} &&\Gamma(t),\\
 -{[\![T\nu_\Gamma]\!]}& = \sigma H_\Gamma \nu_\Gamma &&\hbox{on} &&\Gamma(t),\\   
        u(0)& = u_0        &&\hbox{in} &&\Omega_{0},\\
\end{aligned}
\end{equation}

\begin{equation}
\label{heat}
\begin{aligned}
\varrho\kappa(\theta)(\partial_t \theta  + (u|\nabla \theta)) -{\rm div}(d(\theta)\nabla\theta) &=2\upmu(\theta)|D(u)|_2^2
&& \mbox{in} &&\Omega\setminus \Gamma(t),\\
\partial_\nu \theta &=0 &&\mbox{on} &&\partial\Omega,\\
[\![\theta]\!] &=0 &&\mbox{on} &&\Gamma(t),\\
[\![d(\theta)\partial_\nu\theta]\!]&=0 &&\mbox{on} &&\Gamma(t),\\
\theta(0)&=\theta_0  && \mbox{in} &&\Omega_0,
\end{aligned}
\end{equation}

\begin{equation}
\label{kinematic}
\begin{aligned}
\hspace{2.5cm}
V_\Gamma& = (u|\nu_\Gamma)  && \hbox{on} &&\Gamma(t),\\               
\Gamma(0)& = \Gamma_0.\\
\end{aligned}
\end{equation}
Here,
$(z|w)=\sum_{j=1}^n z_j \bar w_j$ denotes the inner product in $\C^n$ 
for $z,w\in\C^n$,
$$
[\![\phi]\!](t,x)=\lim_{h \rightarrow 0+}\big(\phi(t,x + h\nu_\Gamma(x)) - \phi(t,x - h\nu_\Gamma(x))\big),\quad x \in \Gamma(t),$$
denotes the jump of the quantity $\phi$, defined on the respective
domains $\Omega_i(t)$, across the interface $\Gamma(t)$,
and 
$|D(u)|_2:= ({\rm trace}\,D(u)^2)^{1/2}$
is the Hilbert-Schmidt norm of the (symmetric) matrix $D(u)$.

\medskip
\noindent
For $\Omega$ and the quantities $\psi_i$, $d_i, \upmu_i$ the following regularity and positivity conditions
are assumed:
\begin{equation}
\label{condition-H}
\begin{aligned}
&\partial\Omega\in C^3,\;
\psi_i\in C^3(0,\infty),\;\; d_i,\upmu_i\in C^2(0,\infty),\;\; \\
&\psi_i^{\prime\prime}(s),d_i(s),\upmu_i(s)>0,\, s>0.
\end{aligned}
\end{equation}

\noindent
Given is the initial position $\Gamma_{0}$, the initial velocity 
$u_0:\Omega_0\to \R^n$ and the initial temperature $\theta_0:\Omega_0\to \R,$ 
where $\Omega_0:=\Omega\setminus\Gamma_0$.

The unknowns in \eqref{NS-heat}-\eqref{kinematic} are
the free boundary $\Gamma(t)=\partial\Omega_1(t)$,
the velocity field $u(t,\cdot):\Omega\setminus\Gamma(t) \to \R^{n}$,
the pressure     $\pi(t,\cdot):\Omega\setminus\Gamma(t) \to \R$,
and the temperature  $\theta(t,\cdot):\Omega\setminus\Gamma(t) \to \R$,
where $n\ge 2$.

In case $\upmu$ is constant, 
the Navier-Stokes system \eqref{NS-heat} decouples from the
advection - diffusion equation  \eqref{heat}.
%%%%%%%%%%%%%%%%%%%%%%%%%%%%%%%%%%%%%%

\medskip
%\begin{remark}
In the {\em isothermal} case, that is, in case that $\theta$ is constant,
system~\eqref{NS-heat}-\eqref{kinematic} reduces to
\begin{equation}
\label{NS-iso}
\begin{aligned}
 \varrho\big(\partial_tu+(u|\nabla)u\big)- \upmu\Delta u +\nabla\pi
                         &=  0 &&\hbox{in} &&\Omega\setminus\Gamma(t),\\
{\rm div}\,u & = 0 &&\hbox{in} &&\Omega\setminus\Gamma(t),\\
 u &=0      &&\hbox{on} &&\partial\Omega,\\    
{[\![u]\!]} & = 0  &&\hbox{on} &&\Gamma(t),\\
 -{[\![T\nu_\Gamma]\!]}& = \sigma H_\Gamma \nu_\Gamma &&\hbox{on} &&\Gamma(t),\\   
V_\Gamma& = (u|\nu_\Gamma)  && \hbox{on} &&\Gamma(t),\\               
u(0)& = u_0        &&\hbox{in} &&\Omega_{0},\\
\Gamma(0)& = \Gamma_0,\\
\end{aligned}
\end{equation}
resulting in the {\em isothermal Navier-Stokes problem with surface tension}.
The corresponding one-phase problem is obtained by setting $\varrho_2=\upmu_2=0$
and discarding $\Omega_2$.

\medskip
If $\theta$ is constant and inertia (i.e., the term $\varrho(\partial_t u + (u|\nabla) u))$ is ignored,  
one is left with a quasi-stationary problem, the {\em two-phase Stokes problem with surface tension},
 which generates the \highlight{\em two-phase Stokes flow with surface tension}.
More precisely, this problem reads
\begin{equation}
\label{Stokes-flow}
\begin{aligned}
-\upmu\Delta u +\nabla\pi&=0 &&\mbox{in} &&\Omega\setminus \Gamma(t),\\
{\rm div}\, u&=0 &&\mbox{in} &&\Omega\setminus \Gamma(t),\\
u&=0 &&\mbox{on} &&\partial\Omega,\\
[\![u]\!] &= 0 &&\mbox{on} &&\Gamma(t),\\
-[\![T\nu_\Gamma]\!]  &=\sigma H_\Gamma\nu_\Gamma &&\mbox{on} &&\Gamma(t),\\
V_\Gamma &= (u|\nu_\Gamma) &&\mbox{on} &&\Gamma(t),\\
\Gamma(0)&=\Gamma_0. &&
\end{aligned}
\end{equation}
%\end{remark}
%%%%%%%%%%%%%%%%

The {\em isothermal one-phase Navier-Stokes problem}  has received wide attention
in the last three decades or so. Existence and uniqueness of solutions
for $\sigma=0$, as well as $\sigma>0$, in case that $\Omega_0$ is bounded
%(corresponding to an isolated fluid drop)
has been extensively studied in a series of papers by Solonnikov,
see for instance~\cite{Sol84, Sol86,Sol87b,Sol87a, Sol89, Sol91,Sol03a, 
Sol03b, Sol04a, Sol04b, Sol08}
and Mogilevski{\u\i} and Solonnikov~\cite{MoSo92}.
Results were established in anisotropic Sobolev-Slobodetskii as well as in  H\"older spaces.
Moreover, it was shown in~\cite{Sol87b} that if $\Omega_{0}$ is sufficiently close
to a ball and the initial velocity $u_{0}$ is sufficiently small,
then the solution exists globally and converges to a uniform rigid rotation
of the liquid about a certain axis which is moving uniformly with a constant speed,
see also Padula and Solonnikov~\cite{PaSo02}.

More recently, local and global existence and uniqueness
results (in case that $\Omega_0$ is a bounded domain,
a perturbed infinite layer, or a perturbed half-space)
in  anisotropic Sobolev spaces $W^{2,1}_{q,p}$
%with $2<p<\infty$ and $n<q<\infty$
have been established by Shibata and Shimizu~\cite{ShSh07a, ShSh08, ShSh11b}
for $\sigma=0$ as well as $\sigma>0$.
Additional existence results can be found in Mucha and Zaj{a}czkowski~\cite{MuZa00} and Abels~\cite{Abe05b}.

The motion of a layer of an incompressible, viscous fluid in an ocean of infinite extent,
bounded below by a solid surface and above by a free surface which includes
the effects of surface tension and gravity, was considered by
Allain~\cite{All87},
Beale~\cite{Bea84}, Beale and Nishida~\cite{BeNi84},
Tani~\cite{Tan96a}, and by Tani and Tanaka~\cite{TaTa95}.
If the initial state and the initial velocity are close
to equilibrium, global existence of solutions was proved
in~\cite{Bea84} for $\sigma>0$,
and in~\cite{TaTa95} for $\sigma\ge 0$,
and the asymptotic decay rate for $t\to\infty$
was studied in~\cite{BeNi84}.

Existence and uniqueness of local strong solutions for the 
{isothermal two-phase problem}~\eqref{NS-iso}
was first studied by Denisova~\cite{Den90, Den94} and Denisova and Solonnikov~\cite{DeSo91, DeSo95},
%Densiova~\cite{Den94} establishes existence and uniqueness of solutions (of the transformed problem in %Lagrangian coordinates) with $v\in W^{r,r/2}_2$ for $r\in (5/2,3)$.
while global existence results were established in \cite{Den07, Den14, DeSo11, Sol14}.
Shimizu~\cite{Shi11} obtained existence and uniqueness results in
anisotropic Sobolev $W^{2,1}_{q,p}$-spaces. 

%for $2<p<\infty$ and $n<q<\infty$.

Pr\"uss and Simonett~\cite{PrSi10a, PrSi10b, PrSi11}
considered the two-phase Navier-Stokes equations
with $\sigma>0$ in a situation where the free boundary $\Gamma$ is given
as the graph of a function over a hyperplane, and gravity is acting on the fluids
\cite{PrSi10a, PrSi11}.
It was shown in~\cite{PrSi10b, PrSi11}
that solutions regularize and immediately become real analytic in space and time.
It is well-known that the situation where gravity acts on two superposed immiscible fluids - with the heavier fluid lying above a fluid of lesser density - can lead to an instability, 
the famous {Rayleigh-Taylor instability},
% In this case, small disturbances of an equilibrium state can cause instabilities, where the heavy fluid moves down under the influence of gravity, and the light material is displaced upwards.
see Pr\"uss and Simonett~\cite{PrSi10b}, Wang and Tice~\cite{WaTi12}, and Wilke~\cite{Wil14}
for results in this direction.

K\"ohne, Pr\"uss, and Wilke~\cite{KPW13} obtained existence and uniqueness of strong solutions
with maximal regularity for \eqref{NS-iso}.
It was shown that the equilibrium states are given by
 zero velocities, constant pressures in the phase components,
while $\Omega_1$ consists of a collection of balls (of possibly different) radii.
Moreover, nonlinear stability of equilibria and convergence was established.
A similar result was also obtained in \cite{DeSo11, Sol14} by a different approach.

Tanaka~\cite{Tan95} and Denisova~\cite{Den05} considered the two-phase Navier-Stokes equations
with thermo-capillary convection \eqref{NS-heat}-\eqref{kinematic} and obtained existence and uniqueness of strong local solutions.

Existence and uniqueness of solutions for the quasi-stationary one-phase Stokes flow
was first obtained  by G\"unther, Prokert~\cite{GuPr97, Pro97},
see also Solonnikov~\cite{Sol99}. %and  Escher and Prokert  for regularity results.
It  was shown in ~\cite{GuPr97} that in case the initial domain of a fluid drop is close to a ball,
the solution exits globally and converges to a ball at an exponential rate.
This result was rederived by Friedman and Reitich~\cite{FrRe02b} by a different method.

 A common approach employed by many authors to analyze free boundary problems in
fluid flows relies on the use of Lagrangian coordinates.
%In this formulation one obtains a transformed
%problem for the velocity and the pressure on a fixed domain,
%where the moving boundary does not occur explicitly. 
%The moving boundary is then given by
%$$
%\Gamma(t)=\big\{\xi+\int_{0}^{t}v(\tau,\xi)\,d\tau: \xi\in \Gamma_0\big\},
%$$
%where $v$ is the velocity field in Lagrangian coordinates.
%    With this approach it seems difficult to prove additional regularity properties of solutions,
%    for instance smoothness of the free boundary.
%    In addition, Lagrangian coordinates do not seem well-adapted in the presence of phase transitions.

Here, a different approach is used, namely the  {\em direct mapping method} 
based on the \highlight{\em Hanzawa transform},
which has proven itself to be particularly useful in the study of phase transitions,
and in the study of fluid flows in the presence of phase transitions,
see  for instance 
\cite{PSSS12, PrSh12, PSSW14, PSW14, PSW12, PSZ13a, PSZ13b},
and the monograph \cite{PrSi16}.
This approach was also essential in establishing regularity properties 
of solutions for incompressible fluid flows, see \cite{PrSi10a, PrSi10b, PrSi11} and \cite{PrSi16}.
The direct mapping method has also been employed by Solonnikov in the
recent works \cite{Sol03a, Sol04a, Sol04b, Sol08, Sol14}.

The reader will find a systematic exposition of the Hanzawa transform and its application to 
free boundary problems in Chapter 7.1 of this handbook.

\smallskip
The main results of this chapter concern the stability analysis of equilibria
and the qualitative behavior of global solutions.
It turns out that the equilibrium states for the problems introduced above are not isolated,
but instead give rise to a finite-dimensional smooth manifold $\cE$.
Consequently, the kernel of the linearization $L$ (to be discussed below) 
at an equilibrium is nontrivial and has at least the 
dimension of the manifold $\cE$.

It will be shown that for any equilibrium $e_*\in\cE$, the kernel ${\sf N}(L)$ is
in fact isomorphic to the tangent space of $\cE$ at $e_*$, the eigenvalue
$0$ of $L$ is semi-simple, and the remaining spectral part of the
linearization $-L$ is stable. 
Hence, every equilibrium $e_*$ is \highlight{\em normally stable}. 

It will then be shown that every solution that starts out close to an equilibrium 
exists globally and converges to a (possibly different) equilibrium at an exponential rate.
In a simpler context, this result has been termed  
the {\em generalized principle of linearized stability}, see \cite{PSZ09b, PSZ09a}.
For problem~\eqref{NS-heat}--\eqref{kinematic}, as well as problem~\eqref{NS-iso},
a significant challenge arises as solutions
live on a nonlinear manifold, the {\em state manifold} $\cSM$,
caused by nonlinear compatibility conditions. 

Nonlinear stability for problem \eqref{NS-heat}-\eqref{kinematic} and  problem \eqref{NS-iso} 
will be obtained by an application of the implicit function theorem in
combination with degree theory.

It will be shown that each of the problems introduced above 
possesses a natural (strict) Lyapunov functional.
Consequently, the limit sets
of  solutions are contained in the set of equilibria $\cE$. 
Combining this with compactness of solutions and the local stability of equilibria
one shows that any solution which does not develop singularities converges
to an equilibrium in the topology of the state manifold $\cSM$.

Stability of equilibria for the two-phase isothermal Navier-Stokes problem \eqref{NS-iso}
was established by K\"ohne, Pr\"uss, and Wilke \cite{KPW13}, and 
 Denisova, Solonnikov~\cite{DeSo11, Sol14},
while the corresponding results for the two-phase  Navier-Stokes equations with heat conduction 
\eqref{NS-heat}-\eqref{kinematic}
as well as the two-phase Stokes flow \eqref{Stokes-flow} are contained
in the monograph by Pr\"uss and Simonett~\cite{PrSi16}.

It is interesting to note that similar stability results also hold in
the more complex situation of fluid flows with phase transitions,
with the notable difference that multiple spheres turn out to be unstable in the presence of 
phase transitions.
The reader is once more referred to the monograph \cite{PrSi16} for a comprehensive discussion 
of fluid flows with phase transitions.
%see also  Chapter 7.1.
%%%%%%%%%%%%%%%%%%%%%%%%%%%%%%%%%%%%%%%%%%%%%%%%%
\section{Equilibria, energy, and entropy}\label{sect-existence}
%%%%%%%%%%%%%%%%%%%%%%%%%%%%%%%%%%%%%%%%%%%%%%%%%%
In this section it will be shown that the total energy for problem \eqref{NS-heat}-\eqref{kinematic}
is preserved, while the total entropy is nondecreasing,
implying that the model is thermodynamically consistent.
In addition, it will be shown that equilibrium states correspond to zero velocities,
constant pressures in the components of the phases, constant temperature, 
while the dispersed phase consists of a union of disjoint balls.
%Although these results are implicitly contained in Chapter 7.1,
% a detailed proof is nevertheless provided in order to keep this contribution as self-contained as possible.
\subsection{Local existence}
%%%%%%%%%%%%%%%%%%%%%%%%%%
The basic well-posedness result for problem~\eqref{NS-heat}-\eqref{kinematic} 
reads as follows, where 
 $\cP_\Gamma=I-\nu_\Gamma\otimes \nu_\Gamma$ denotes the orthogonal projection onto the tangent space of $\Gamma$.
%%%%%%%%%%%%%%%%%%
\begin{theorem}
\label{thm:local-ex}
Let $p>n+2$ and suppose that condition \eqref{condition-H} holds.
Assume the regularity conditions
\begin{equation*}
(u_0,\theta_0)\in W^{2-2/p}_p(\Omega\setminus\Gamma_0)^{n+1},\quad \Gamma_0\in W^{3-2/p}_p,
\end{equation*}
the compatibility conditions
\begin{equation*}
\begin{aligned}
&{\rm div}\, u_0=0 \;\;{\rm in}\;\; \Omega\setminus\Gamma_0,
\quad u_0=0\;\;{\rm and}\;\; \partial_\nu\theta_0=0\;\; {\rm on}\;\; \partial\Omega,\\
& [\![u_0]\!]=0,\;\; 
\cP_{\Gamma_0}[\![\upmu(\theta_0) D(u_0) \nu_{\Gamma_0}]\!]=0\;\;
{\rm on}\;\; \Gamma_0, \\
&[\![\theta_0 ]\!]=0,\;\; [\![d(\theta_0)\partial_\nu\theta_0]\!]=0\;\; {\rm on}\;\; \Gamma_0,
\end{aligned}
\end{equation*}
and  the well-posedness condition $\theta_0>0$ in $\bar\Omega$. 
\\
Then there exists a number $a=a(u_0,\theta_0,\Gamma_0)$ 
and a unique classical solution $(u,\pi,\theta,\Gamma)$
of \eqref{NS-heat}-\eqref{kinematic} on the time interval $(0,a)$.
Moreover, 
${\mathcal M}=\bigcup_{t\in (0,a)}\{t\}\times \Gamma(t)$
is real analytic, provided  the functions $\psi_i$, $\upmu_i$, and $d_i$ share this property.
\end{theorem}
\begin{proof}
For a proof the reader is referred to \cite[Sections 9.2 and 9.4]{PrSi16}.
\end{proof}
%\goodbreak
%%%%%%%%%%%%%%%%%%%
\subsection{The state manifold}
%%%%%%%%%%%%%%%%%%%%%%%%%%%%
It can be shown that the closed  $C^2$-hypersurfaces contained in $\Omega$ which bound a region 
$\Omega_1\subsubset\Omega$ 
form a $C^2$-manifold, denoted by $\cMH^2(\Omega)$, see for instance \cite{PrSi13}, or \cite[Chapter 2]{PrSi16}.

The charts are the parameterizations over a given hypersurface $\Sigma$, and the tangent
space consists of the normal vector fields on $\Sigma$.

\medskip

The \highlight{state manifold} $\cSM$ for problem \eqref{NS-heat}-\eqref{kinematic} is defined by
%%%%%%%%%%%%%%%%%%%
\begin{eqnarray}
\label{def-SM}
\cSM:=&&\hspace{-0.5cm}\Big\{(u,\theta,\Gamma)\in C(\bar{\Omega})^{n+1}\times \cMH^2(\Omega): 
 \nonumber\\
 &&  (u,\theta)\in W^{2-2/p}_p(\Omega\setminus\Gamma)^{n+1},\;\; \Gamma\in W^{3-2/p}_p,\nonumber\\
 && {\rm div}\, u=0\;\; \mbox{in}\;\; \Omega\setminus\Gamma, \quad \theta>0\;\;\mbox{in}\;\;\bar\Omega, 
 \quad  u=0,\;\;\partial_\nu\theta=0\;\; \mbox{on}\;\; \partial\Omega,\nonumber\\
 &&  \cP_\Gamma\, [\![ \upmu (\theta)D(u)\nu_\Gamma]\!]=0, \;\;
     [\![d(\theta)\partial_\nu \theta ]\!]=0 \;\; \mbox{on}\;\;\Gamma
\Big\}.\nonumber
\end{eqnarray}
Charts for this manifold are obtained by the charts induced by $\cMH^2(\Omega)$,
followed by a Hanzawa transformation.

Applying Theorem \ref{thm:local-ex} and re-parameterizing the interface repeatedly,
one shows that \eqref{NS-heat}-\eqref{kinematic} yields a local semiflow on $\cSM$.

It is noticeable that the pressure $\pi$ does not explicitly occur as a variable 
in the definition of the state manifold $\cSM$.
In fact,  $\pi$ is determined at each time $t$ from
$(u,\theta,\Gamma)$ by means of the weak  transmission problem
%%%%%%%%%%%%%%%%%%%%%%
\begin{equation*}
\begin{aligned}\label{weak-transmission}
\left.\left(\varrho^{-1}\nabla{\pi}\,\right|\nabla\phi\right)_{L_2(\Omega)}&=
\left.\left( 2\varrho^{-1} {\rm div}\,(\upmu(\theta)D(u))-(u|\nabla) u\,\right|\nabla\phi\right)_{L_2(\Omega)},
\quad \phi\in H^{1}_{p^\prime}(\Omega),\\
[\![\pi]\!]&= \sigma H_\Gamma + ([\![2\upmu(\theta)D(u)\nu_\Gamma ]\!]|\nu_\Gamma)
\quad\mbox{on}\;\;\Gamma.
\end{aligned}
\end{equation*}
%%%%%%%%%%%%%%%%%%%%%%%%
Concerning such transmission problems the reader is referred to 
\cite[Theorem 8.5]{KPW13} or \cite[Proposition 8.6.2]{PrSi16}.  
%%%%%%%%%%%%%%%%%%%
\subsection{Conservation of phase volumes}
%%%%%%%%%%%%%%%%%%%%
Suppose that the dispersed phase $\Omega_1$ consists of $m$ disjoint connected components,
 $\Omega_1=\bigcup_{k=1}^m \Omega_{1,k}$. Let
$\Gamma_{\!k}:=\partial\Omega_{1,k}$, $\Gamma=\bigcup_{k=1}^m\Gamma_{\!k}$,
and let
\begin{equation}
{\sf M}_k:=|\Omega_{1,k}|,\quad k=1,\cdots,m,
%:=\int_{\Omega_{1,k}}1\,dx
\end{equation}
denote the volume of $\Omega_{1,k}$. 
Then for any  (sufficiently smooth) solution $(u,\pi,\theta,\Gamma)$ of 
problem~\eqref{NS-heat}-\eqref{kinematic} one obtains, see for instance \cite[Chapter 2]{PrSi16},
\begin{equation*}
\begin{aligned}
\frac{d}{dt}|\Omega_{1,k}(t)|=\int_{\Gamma_{\!k}} V_\Gamma\,d\Gamma_{\!k} 
= \int_{\Gamma_{\!k}} (u|\nu_\Gamma)\,d\Gamma_{\!k} 
=\int_{\Omega_{1,k}}{\rm div}\,u\,dx =0
\end{aligned}
\end{equation*}
for $k=1,\cdots,m$.
This shows that problem \eqref{NS-heat}-\eqref{kinematic} preserves the volume of each
individual phase component. 
%%%%%%%%%%%%%%%%%%%%%
\subsection{Conservation of energy}
The \highlight{\em total energy} for problem~\eqref{NS-heat}-\eqref{kinematic} is defined by
$${\sf E}:= {\sf E}(u,\theta,\Gamma):=
\frac{1}{2} \int_{\Omega\setminus\Gamma} \varrho |u|^2\,dx 
+\int_{\Omega\setminus\Gamma}\varrho\epsilon(\theta)\,dx  +\sigma |\Gamma|, $$
where $|\Gamma|$ denotes the surface area of $\Gamma$.
For the time derivative of ${\sf E}$ one obtains
\begin{align*}
&\frac{d}{dt}{\sf E} = \int_\Omega \{\varrho(\partial_t u|u) + \varrho\partial_t \epsilon(\theta)\}\,dx
-\int_\Gamma \{[\![\frac{\varrho}{2}|u|^2 +\varrho \epsilon(\theta)]\!]+\sigma H_\Gamma\}V_\Gamma\,d\Gamma\\
& = -\int_\Omega  \{\varrho((u|\nabla)u|u) - ({\rm div}\,T | u)
+\varrho (u|\nabla \epsilon(\theta))  - {\rm div}\, (d(\theta)\nabla \theta) - 2\upmu(\theta)|D(u)|_2^2\}\,dx\\
&\quad-\int_\Gamma\big \{[\![\frac{\varrho}{2}|u|^2 +\varrho \epsilon(\theta)]\!]
+\sigma H_\Gamma\big\}V_\Gamma\, d\Gamma\\
&=- \int_\Gamma\big\{
[\![d(\theta)\partial_\nu\theta ]\!] + [\![(Tu|\nu_\Gamma) ]\! ] +\sigma H_\Gamma V_\Gamma\big\}\,d\Gamma \\
&=- \int_\Gamma\big\{
 [\![d(\theta)\partial_\nu\theta ]\!] + ([\![T\nu_\Gamma ]\!] +\sigma H_\Gamma \nu_\Gamma|u)\big\}\,d\Gamma=0,
\end{align*}
implying that the total energy is preserved.
%%%%%%%%%%%%%%%%%%%%%%
\subsection{Entropy and equilibria}
%%%%%%%%%%%%%%%%%%%%%
The \highlight{\em total entropy} for \eqref{NS-heat}-\eqref{kinematic} is defined by
\begin{equation*}
\label{entropy}
\Phi(\theta,\Gamma)=\int_{\Omega\setminus\Gamma} \varrho\eta(\theta)\,dx.
\end{equation*}
With the relation $\epsilon^\prime(\theta)=\theta\eta^\prime(\theta)$ one obtains
\begin{align*}
\frac{d}{dt}& \Phi(t) = \int_\Omega \varrho \partial_t \eta(\theta)\,dx
 - \int_\Gamma [\![\varrho \eta(\theta)]\!] V_\Gamma\, d\Gamma\\
&= \int_\Omega\varrho\eta^\prime(\theta)\partial_t\theta\,dx
- \int_\Gamma [\![\varrho \eta(\theta)]\!] V_\Gamma\,d\Gamma\\
&= \int_\Omega\Big(\frac{1}{\theta}\Big\{2\upmu(\theta)|D(u)|_2^2 + {\rm div}\,(d(\theta)\nabla\theta)\Big\}
-\varrho (u|\nabla \eta(\theta))\Big)\,dx
- \int_\Gamma [\![\varrho \eta(\theta)]\!] V_\Gamma\, d\Gamma\\
&= \int_\Omega 
\Big(\Big\{\frac{2\upmu(\theta)|D(u)|_2^2}{\theta}+\frac{d(\theta)|\nabla\theta|^2}{\theta^2}\Big\} 
+ {\rm div}\,\Big(\frac{d(\theta)}{\theta}\nabla\theta \Big)-\varrho (u|\nabla \eta(\theta))\Big)\,dx \\
&\quad - \int_\Gamma [\![\varrho \eta(\theta)]\!] V_\Gamma\, d\Gamma\\
&= \int_\Omega \Big\{\frac{2\upmu(\theta)|D(u)|^2_2}{\theta}+\frac{d(\theta)|\nabla\theta|^2}{\theta^2}\Big\}\,dx.
\end{align*}
Hence the total entropy is nondecreasing, and this shows that the model is thermodynamically consistent.

\medskip
\noindent
It follows that the negative entropy 
is a Lyapunov functional for system \eqref{NS-heat}-\eqref{kinematic}.

Even more, $-\Phi$ is a strict Lyapunov functional.
To see this, assume that $\Phi$ is constant on some interval $(t_1,t_2)$. Then $d\Phi/dt=0$ in $(t_1,t_2)$, hence $D(u)=0$ and $\nabla\theta=0$ in $(t_1,t_2)\times\Omega$.
Therefore, $\theta$ is constant on connected components, and as $\theta$ is continuous,
it follows that $\theta=\theta_\ast$ on $\Omega$, with $\theta_\ast$ a constant.
Next, by $[\![u]\!]=0$ and Korn's inequality we  have $\nabla u=0$, 
and then $u=0$ by the no-slip condition on $\partial\Omega$.
 This implies further $(\partial_t\theta,\partial_t u)=0$ and $V_\Gamma=0$, i.e., the system is at equilibrium.
 Furthermore, $\nabla\pi=0$, and consequently the pressure is constant in the components of the bulk phases. 
 Therefore,
 $\sigma H_\Gamma=[\![\pi]\!]$ is constant on each component $\Gamma_{\!k}$ of the interface $\Gamma$.
 This implies that the dispersed phase $\Omega_1$ is a ball if it is connected, 
 or a collection $\bigcup B(x_k,R_k)$ of balls, with the radii 
 $R_k$ of the balls related to the pressures by the \highlight{\em Young-Laplace law}
 \begin{equation}
 \label{Young-Laplace}
 [\![\pi]\!]\Big|_{\Gamma_{\!k}}=\sigma H_{\Gamma_{\!k}} = -\frac{\sigma (n-1)}{R_k}.
 \end{equation}
\begin{remark}
{\bf (i)} It is noted that at equilibrium the dispersed phase consists of at most
countably many disjoint balls $B(x_k, R_k)$. If there are infinitely many of them,
then $R_k\to0$ as $k\to\infty$, hence the corresponding curvatures $H_{\Gamma_{\!k}}= -(n-1)/R_k$
tend to infinity, and so do the pressures inside these balls. This is due to the model
assumption that there are no phase transitions. On the other hand, phase transitions will
occur at very high pressure levels.
Therefore, although thermodynamically consistent, the model \eqref{NS-heat}-\eqref{kinematic} is physically not very realistic. 

To avoid this contradiction,
in the sequel only equilibria in which the dispersed phase consists
of finitely many balls are considered. Note also that the free boundary will not
be of class $C^2$ if $\Omega_1$ has infinitely many components.

\medskip
\goodbreak
{\bf (ii)} There is another pathological case which will be excluded in the sequel, 
namely the one where the dispersed phase contains balls touching each other. This can only happen if the radii
of these balls are equal, as the pressure jump would otherwise not be constant on $\Gamma$.
Physically one would expect such an equilibrium to be unstable.
Observe that also in this situation the free boundary $\Gamma$ is not a manifold of class~$C^2$.
\end{remark}
%%%%%%%%%%%%%%%%%%%%%%
\subsection{The manifold of equilibria}
As shown above the equilibrium states of system \eqref{NS-heat}-\eqref{kinematic} 
are  zero velocities $u_*=0$, constant pressures $\pi_*$ in the phases, constant temperature $\theta_*$, 
and the dispersed phase $\Omega_1$ consists of a collection of balls.
An equilibrium is called {\em non-degenerate} if
\begin{itemize}
\item[{\bf (i)}] $\Omega_1$ consists of a {\em finite} collection of balls;
\vspace{2mm}
\item[{\bf (ii)}] the balls do not touch the outer boundary $\partial\Omega$ and do not touch each other.
\end{itemize}
%\begin{equation}
%\label{non-degenerate}
%\begin{aligned}
%&\quad {\bf (i)}\;\; \mbox{$\Omega_1$ consists of a {\em finite} collection of balls;} \\
%\vspace{2mm}
%&\quad {\bf (ii)}\:  \mbox{the balls do not touch the outer boundary $\partial\Omega$ and do not touch each other.}\\
%\end{aligned}
%\end{equation}
Suppose then that $\Omega_1=\bigcup_{k=1}^m B(x_k,R_k)$, with $B(x_k,R_k)$ a ball with center $x_k$ 
and radius $R_k$.
Then 
\begin{equation}
\label{def-S}
\begin{aligned}
\cS:=\Big\{\Sigma=\bigcup_{k=1}^m\Sigma_k:\; &\Sigma_k=\partial B(x_k,R_k), \;\;  \bar B(x_k,R_k)\subset\Omega, \\
&\hspace{1cm} \bar B(x_k,R_k)\cap\bar B(x_j,R_j)=\emptyset\;\;\mbox{for}\;\; j\neq k \Big\}
\end{aligned}
\end{equation}
forms a smooth (in fact analytic) manifold of dimension $m(n+1)$.
To verify this, it will be shown how a neighboring sphere can be parameterized over a given one.
Let us assume that $\Sigma=\partial B(0,R)$ is centered at the origin of $\R^n$.
Suppose ${S}\subset \Omega$ is a sphere that is
sufficiently close to $\Sigma$ and denote by $(y_1,\cdots,y_n)$ the
coordinates of its center and let $y_0$ be such that $R+y_0$
corresponds to its radius. 
One verifies that
$(R+y_0)^2=\sum_{j=1}^n ((R+\delta)Y^j-y_j)^2,$ where
$\delta$ denotes the signed distance function with respect to $\Sigma$,
and  $Y^j$ are the spherical harmonics of degree one on $\Sigma$.
Using the relation  $\sum_{j=1}^n (Y^j)^2=1$ and solving the quadratic equation for $R+\delta$ 
shows that  ${S}$ can be parameterized over $\Sigma$ (in normal direction) by 
\[
\delta(y_0,y_1,\cdots,y_n)=\sum_{j=1}^n y_j Y^j-R+\sqrt{(\sum_{j=1}^n y_j
Y^j)^2+(R+y_0)^2-\sum_{j=1}^n y_j^2}.
\]
Clearly, the mapping $z\mapsto \delta(z)$ provides a real analytic parameterization.
The case of $m$ spheres can be treated analogously,
yielding $m(n+1)$ degrees of freedom.
Finally, 
\begin{equation}
\label{def-E}
\cE=\{(0,\theta_*,\Sigma):\; \mbox{$\theta_*>0$ is constant},\;\; \Sigma\in\cS\} 
\end{equation}
denotes the set of all non-degenerate equilibria of \eqref{NS-heat}-\eqref{kinematic}.
It follows that $\cE$ is a real analytic manifold of dimension $m(n+1)+1$.
Here we note that the (constant) pressures at equilibrium can be determined by the Young-Laplace law~\eqref{Young-Laplace}. 
%%%%%%%%%%%%%%%%%%%%%%
\subsection{Equilibria as critical points of the entropy}
%%%%%%%%%%%%%%%%%%%%%
%Another interesting observation is the following.
Suppose that $\Omega_1$ is the union of $m$ disjoint components $\Omega_{1,k}$
and $\Gamma=\bigcup_{k=1}^m \Gamma_{\!k}$ with $\Gamma_{\!k}=\partial\Omega_{1,k}$.

The goal of this subsection is to determine the critical points of the total entropy $\Phi$ under the constraints of 
given total energy ${\sf E}={\sf E}_0$
and given phase volumes ${\sf M}_k={\sf M}_{0,k}$, say on $C(\bar{\Omega})^{n+1}\times\cMH^2(\Omega)$.
See Section 2.2 for the definition of $\cMH^2(\Omega)$.

\medskip
With $\Phi(\theta,\Gamma)=\int_{\Omega\setminus\Gamma} \varrho\eta(\theta)\,dx$ and
$$ {\sf M}_k(\Gamma)=|\Omega_{1,k}|,
\quad  {\sf E(u,\theta,\Gamma)}=\int_{\Omega\setminus\Gamma} \{(\varrho/2)|u|^2 + \varrho\epsilon(\theta)\}\,dx + \sigma |\Gamma|,$$
the method of Lagrange multipliers yields
$$ \Phi^\prime (\theta,\Gamma) +\sum_{k=1}^m\lambda_k {\sf M}_k^\prime(\Gamma) 
+ \mu {\sf E}^\prime(u,\theta,\Gamma) =0$$
at a critical point $e=(u,\theta,\Gamma)$.
We will now compute 
the derivatives of these functionals
in the direction of
$$z=(v, \vartheta, h_1,\cdots,h_m),$$
where each component $\Gamma_k$ is varied in the direction of the normal vector 
field $h_k \nu_{\Gamma_{\!k}}$, with $h_k:\Gamma_{\!k}\to\R$ a (sufficiently smooth) function.
This yields
\begin{align*}
\langle\Phi^\prime(\theta,\Gamma)|z\rangle &= \int_\Omega\varrho \eta^\prime(\theta) \vartheta \,dx
-\sum_{k=1}^m\int_{\Gamma_{\!k}} [\![\varrho\eta(\theta)]\!]h_k \,d\Gamma_{\!k},\\
\langle {\sf M}_k^\prime(\Gamma) | z\rangle &= \int_{\Gamma_{\!k}} h_k\,d\Gamma_{\!k},\qquad k=1,\cdots,m, \\
\langle {\sf E}^\prime(u,\theta,\Gamma) | z\rangle
&= \int_\Omega \!\{\varrho (u|v) \!+\!\varrho \epsilon^\prime(\theta) \vartheta \}\,dx 
-\!\sum_{k=1}^m \!\int_{\Gamma_{\!k}}\! \{[\![(\varrho/2)|u|^2 \!\!+\!\varrho \epsilon(\theta)]\!]  
\!+\! \sigma H_{\Gamma_{\!k}}\}h_k\,d\Gamma_{\!k}.
\end{align*}
Varying $\vartheta$ first, while setting all other variables of $z$ equal to zero,  yields
$$ \varrho \eta^\prime(\theta) +\mu \varrho \epsilon^\prime(\theta) =0, $$  
hence $\epsilon^\prime(\theta)=\theta \eta^\prime(\theta) = \kappa(\theta)>0$ 
implies $\theta=-1/\mu>0$ constant. Next we vary $v$ to obtain $u=0$ as $\mu\neq0$.
 Finally, each component $h_k$ is varied separately, to the result
$$-[\![\varrho\eta(\theta)]\!] + \lambda_k  +([\![\varrho\epsilon(\theta)]\!] 
+ \sigma H_{\Gamma_{\!k}})/\theta=0
$$
on $\Gamma_{\!k}$, which by the definition of $\epsilon$ yields
\begin{equation}
\label{lambda-k}
 [\![\varrho \psi(\theta)]\!] + \sigma H_{\Gamma_{\!k}} = -\lambda_k \theta \;\; \mbox{on}\;\; \Gamma_{\!k}.
\end{equation}
This implies that $H_{\Gamma_{\!k}}$ is constant for each $k\in\{1,\cdots,m\}$, hence $\Omega_1$ consists 
of a finite collection of balls.

In summary, the critical points $e=(0,\theta,\Gamma)$ 
of the total entropy with the constraints of prescribed phase volumes 
and prescribed total energy are precisely the equilibria of the system.

\medskip
\goodbreak
It is also interesting to note that each equilibrium $e=(0,\theta,\Gamma)$ with $\Gamma=\bigcup_{k=1}^m\Gamma_{\!k}$
is a local maximum of the entropy  w.r.t.\ the constraints
${\sf E}={\sf E}_0$ and ${\sf M_k}={\sf M}_{0,k}$ constant. 
In order to verify this one needs to show that
$$\cD(e):=[\Phi+ \sum_{k=1}^m \lambda_k {\sf M}_k +\mu {\sf E}]^{\prime\prime}(e)$$ 
is negative definite on 
${\rm ker}\,({\sf E}^\prime(e))\bigcap_{k=1}^m {\rm ker}\,({\sf M}^\prime_k(e)),$
the intersection of the kernels of ${\sf E}^\prime(e)$ and  ${\sf M}_k^\prime (e)$, 
where $(\lambda_1,\cdots,\lambda_m,\mu)$ are the fixed Lagrange multipliers found above.
The kernel of ${\sf M}_k^\prime(e)$ is given by
\begin{equation}\label{kM}
\int_{\Gamma_{\!k}} h_k \,d\Gamma_{k} =0, \quad k=1,\cdots,m,
\end{equation}
and that of ${\sf E}^\prime(e)$ by
\begin{equation}\label{kE}
\int_\Omega\varrho\kappa(\theta)\vartheta \,dx
=\sum_{k=1}^m \big([\![\varrho\epsilon(\theta)]\!]+\sigma H_{\Gamma_{\!k}}\big)\int_{\Gamma_{\!k}}h_k\,d\Gamma_{\!k}.
\end{equation}
A straightforward calculation yields
\begin{align}\label{2var}
-\theta\langle \cD(e) z|z\rangle &= \int_\Omega \varrho|v|^2\,dx
+ \int_\Omega(\varrho\kappa(\theta)/\theta)\vartheta^2\,dx 
- \sigma \sum_{k=1}^m\int_{\Gamma_{\!k}} (H_{\Gamma_{\!k}}^\prime h_k)h_k \,d\Gamma_{\!k}.\nn
\end{align}
It is well-known that the linearization of $H_{\Gamma_k}$ at a sphere $\Gamma_k$ is given by
$$H_{\Gamma_{\!k}}^\prime= (n-1)/R_k^2+\Delta_{\Gamma_{\!k}},$$ 
where  $R_k$ is the radius of the sphere  and $\Delta_{\Gamma_{\!k}}$ 
denotes the Laplace-Beltrami operator on $\Gamma_{\!k}$,
see for instance \cite{PrSi13}, or \cite[Chapter 2]{PrSi16}.
As $\varrho$, $\kappa$ are positive and $H^\prime_{\Gamma_{\!k}}$ is negative semi-definite
for all functions in $L_2(\Gamma_{\!k})$ with mean value zero,
one concludes that the form 
$\langle \cD(e) z|z\rangle$ is indeed negative definite on $L_2(\Gamma)$.  
\medskip

\noindent
An  equilibrium $e$ is generically not isolated. 
If a sphere $\Gamma_{\!k}$ does not touch the outer boundary, 
it may be moved inside of $\Omega$ without changing the total entropy. 
This fact is reflected in $\cD(e)$ by choosing $(v,\vartheta)=0$ and $h_k=Y_k^j$, 
the spherical harmonics for $\Gamma_{\!k}$, which satisfy $H_{\Gamma_{\!k}}^\prime Y_k^j=0$.

\medskip
\noindent
Summarizing, the following result has been established.
%%%%%%%%%%%%%%%%%%%%%%
\begin{theorem}
The following assertions hold for problem \eqref{NS-heat}-\eqref{kinematic}.
\begin{itemize}
\vspace{1mm} \item
[{\bf (a)}] The phase volumes $|\Omega_{1,k}|$ and the total energy ${\sf E}$ are preserved. 
\vspace{1mm} \item
[{\bf (b)}] The negative total entropy $-\Phi$ is a strict Lyapunov functional.
\vspace{1mm}\item
[{\bf (c)}] The {\em non-degenerate equilibria} are zero velocities, constant pressures in the
components of the phases, constant temperature, and the interface is a finite union of non-intersecting spheres which 
do not touch the outer boundary $\partial\Omega$. 
\vspace{1mm} \item
[{\bf (d)}] The set $\cE$ of non-degenerate equilibria forms a real analytic manifold
of dimension $m(n+1)+1$, where $m$ denotes the number of connected components of $\Omega_1$.
\vspace{1mm} \item
[{\bf (e)}] The critical points of the entropy functional
for prescribed phase volumes and prescribed total energy are precisely the equilibria of the system. 
\vspace{1mm} \item
[{\bf (f)}] All critical points  of the entropy functional
for prescribed phase volumes and prescribed total energy are local maxima.
\end{itemize}
\end{theorem}
%%%%%%%%%%%%%%%%%%%%%%%%%%%%%%%%%%%%%%%%%%%%%%%%%%
\section{Linear stability}\label{linear stability of equilibria}
%%%%%%%%%%%%%%%%%%%%%%%%%%%%%%%%%%%%%%%%%%%%%%%%%
By employing the Hanzawa transformation, see for instance \cite{PrSi16},
one shows that the pertinent linear problem associated to \eqref{NS-heat}-\eqref{kinematic} 
at an equilibrium $(0,\theta_*,\Sigma)$, with 
$\Sigma=\bigcup_{k=1}^m\Sigma_k$ and $\Sigma_k=\partial B(x_k,R_k)$, is given by

\begin{equation}
\begin{aligned}\label{lin-u}
\varrho\partial_t u-\upmu_* \Delta u +\nabla\pi &=\varrho f_u && \mbox{in} && \Omega\setminus{\Sigma},\\
{\rm div}\, u &=g_d && \mbox{in} && \Omega\setminus{\Sigma},\\
u&=0\quad && \mbox{on} && \partial\Omega,\\
[\![u]\!] &= 0      && \mbox{on} && {\Sigma},\\
 -[\![T_*\nu_\Sigma]\!] +\sigma (\cA_\Sigma h)\nu_\Sigma &=g_u && \mbox{on} && {\Sigma},\\
u(0)&=u_0 &&\mbox{in} && \Omega,
\end{aligned}
\end{equation}

\begin{equation}
\begin{aligned}\label{lin-theta}
\hspace{1.4cm}\varrho\kappa_*\partial_t \vartheta -d_*\Delta \vartheta &= \varrho \kappa_* f_\theta && \mbox{in} && \Omega\setminus{\Sigma},\\
\partial_\nu\vartheta&=0 && \mbox{on} && \partial\Omega,\\
[\![\vartheta]\!]&=0 && \mbox{on} && {\Sigma},\\
-[\![d_*\partial_\nu \vartheta]\!]&= g_\theta && \mbox{on} && {\Sigma},\\
\vartheta(0)&=\vartheta_0  && \mbox{in} && \Omega\setminus\Sigma,
\end{aligned}
\end{equation}

\begin{equation}
\begin{aligned}\label{lin-h}
\hspace{.7cm}
\partial_t h- (u|\nu_\Sigma ) &= f_h  &&\mbox{on} && \Sigma,\\
h(0)&=h_0 &&\mbox{on} && \Sigma,\\
\end{aligned}
\end{equation}
%%%%%%%
where $\vartheta=\theta -\theta_*$ is the relative temperature, $\upmu_*=\upmu(\theta_*)$, 
$\kappa_*=\kappa(\theta_*)$, $d_*= d(\theta_*)$, 
$T_*=2\upmu_*D(u)-\pi I,$
and $h$ is the height function used to parameterize $\Gamma(t)$ over the reference manifold 
$\Sigma$ by means of
$\Gamma(t)=\{q+h(t,q)\nu_\Sigma(p): q\in\Sigma,\; t\ge 0\}.$
Finally, 
the linear operator $\cA_\Sigma$ is given by
$$\cA_\Sigma\big|_{\Sigma_k}=\cA_{\Sigma_k}  = -H_{\Sigma_k}^\prime(0) = -\frac{n-1}{R^2_k} -\Delta_{\Sigma_k}.$$
It is well-known that $\cA_{\Sigma_k}$ is self-adjoint, positive semi-definite
on functions with zero mean, and has compact resolvent in $L_2(\Sigma_k)$.
$\lambda_0=0$ is an eigenvalue with eigenspace of dimension $n$, spanned by the
spherical harmonics of degree one.
$\lambda_{-1}=-(n-1)/R^2_k$ is also an eigenvalue, its eigenspace is one-dimensional
and consists of the constants.

In order to introduce a functional analytic setting for the linear problem~\eqref{lin-u}-\eqref{lin-h}
let
$$X_0=L_{p,\sigma}(\Omega)\times L_p(\Omega)\times W^{2-1/p}_p(\Sigma),$$
 where the subscript $\sigma$ means solenoidal, and define the operator $L$ by
$$ L(u,\vartheta,h)= \big(-(\upmu_*/\varrho)\Delta u +\nabla\pi/\varrho, -(d_*/\varrho\kappa_*)\Delta \vartheta, 
-(u|\nu_\Sigma)\big).$$
The domain $X_1:={\sf D}(L)$ of $L$ is given by
\begin{equation*}
\begin{aligned}
{\sf D}(L)&=\{(u,\vartheta,h)\in H^2_p(\Omega\setminus\Sigma)^{n+1}\times W^{3-1/p}_p(\Sigma): {\rm div}\, u=0\; \mbox{ in }\; \Omega\setminus\Sigma,\\ 
&\hspace{0.7cm} u,\partial_\nu\vartheta=0\;\mbox{ on }\;\partial\Omega,\;\;
[\![u]\!],[\![\vartheta]\!],\cP_\Sigma[\![2\upmu_*  D(u)\nu_\Sigma]\!],[\![d_*\partial_\nu\vartheta]\!]=0\;\:\mbox{on}\:\; \Sigma
\},
\end{aligned}
\end{equation*}
where $\cP_\Sigma$ denotes the orthogonal projection onto the tangent space of $\Sigma$.
Here $\pi\in \dot H^1_p(\Omega\setminus\Sigma)$ is determined 
by means of the weak transmission problem
\begin{equation*}
\begin{aligned}
(\varrho^{-1}\nabla\pi|\nabla\phi)_{L_2(\Omega)} 
&=(\varrho^{-1}\upmu_*\Delta u|\nabla \phi)_{L_2(\Omega)},\quad \phi\in {H}^1_{p^\prime}(\Omega),\\
[\![\pi]\!] &=-\sigma \cA_\Sigma h + ([\![2\upmu_* D(u)\nu_\Sigma ]\!]\,|\nu_\Sigma)\quad \mbox{ on } \Sigma.
\end{aligned}
\end{equation*}
Setting $z=(u,\vartheta,h)$, $f=(f_u,f_\theta,f_h)$,  
and $g=(g_d, g_u,g_\theta)$ one concludes that 
the linear problem~\eqref{lin-u}-\eqref{lin-h} can be rewritten as a Cauchy problem
$$\dot z + Lz=f(t),\quad z(0)=z_0,$$
provided $g=0$.
Associated with the operator $L$ is the eigenvalue problem
\begin{equation}
\begin{aligned}\label{ev-lin-u}
\varrho\lambda  u-\upmu_* \Delta u +\nabla\pi &= 0 && \mbox{in} && \Omega\setminus{\Sigma},\\
{\rm div}\, u &= 0 && \mbox{in} && \Omega\setminus{\Sigma},\\
u&=0  && \mbox{on} && \partial\Omega,\\
[\![u]\!] &= 0      && \mbox{on} && {\Sigma},\\
 -[\![T_*\nu_\Sigma]\!] +\sigma (\cA_\Sigma h)\nu_\Sigma &= 0 && \mbox{on} && {\Sigma},\\
\lambda  h- (u|\nu_\Sigma) &= 0 &&  \mbox{on} && {\Sigma},\\
\end{aligned}
\end{equation}
and
%%%%
\begin{equation}
\begin{aligned}\label{ev-lin-theta}
\hspace{1.2cm}\varrho\kappa_*\lambda \vartheta -d_*\Delta \vartheta &=0 && \mbox{in} && \Omega\setminus{\Sigma},\\
\partial_\nu\vartheta&=0 && \mbox{on} && \partial\Omega,\\
[\![\vartheta]\!]&=0 && \mbox{on} && {\Sigma},\\
-[\![d_*\partial_\nu \vartheta]\!]&= 0 && \mbox{on} && {\Sigma}.\\
\end{aligned}
\end{equation}
Note that the eigenvalue problems \eqref{ev-lin-u} and \eqref{ev-lin-theta}
decouple.
\begin{theorem}
\label{linstab-heat} 
Let $e_*\in\cE$ be an equilibrium. Then $L$ has the following properties.
\begin{itemize}
\vspace{1mm}
\item[{\bf (a)}]
$-L$ generates a compact, analytic $C_0$-semigroup in $X_0$ which has 
the property of maximal $L_p$-regularity. 
\vspace{1mm}
\item[{\bf (b)}]
The spectrum of $L$ consists of countably many eigenvalues of finite algebraic multiplicity. 
\vspace{1mm}
\item[{\bf (c)}] $-L$ has no eigenvalues $\lambda\neq0$ with nonnegative real part. 
\vspace{1mm}
\item[{\bf (d)}] $\lambda=0$ is a semi-simple eigenvalue of $L$ of multiplicity $m(n+1)+1$. 
\vspace{1mm}
\item[{\bf (e)}] The kernel ${\sf N}(L)$ of $L$ is isomorphic to the tangent space $T_{e_*}\cE$ of 
%the manifold of equilibria 
$\cE$ at $e_*$.
\end{itemize}
Hence, the equilibrium $e_*\in\cE$ is normally stable.
\end{theorem}
%%%%%%%%%%%%%
\begin{proof}
(a) It follows from the results in \cite{PrSi16} that
$-L$ generates a compact, analytic, strongly continuous semigroup in $X_0$ which has maximal $L_p$-regularity.

\medskip
\noindent
(b)
As $L$ has compact resolvent,
the spectrum of $L$ consists entirely of eigenvalues of finite algebraic multiplicity.

\medskip
\noindent
(c) 
Suppose that $\lambda\neq 0$ with ${\rm Re}\; \lambda\geq0$ is an eigenvalue of $-L$
with eigenfunction $(u,\vartheta,h)$.
Since \eqref{ev-lin-u}-\eqref{ev-lin-theta} decouple, 
either $(u,h)\neq (0,0)$ or $\vartheta\neq 0$.

Let us first assume that $(u,h)\neq (0,0)$. 
Taking the $L_2$-inner product of the equation for $u$ with $u$ and integrating by parts yields
\begin{equation}\label{evid-1}
\begin{aligned}
0&=\lambda |\varrho^{1/2}u|^2_{L_2(\Omega)} -({\rm div}\; T_*|u)_{L_2(\Omega)}\\
&=\lambda |\varrho^{1/2}u|^2_{L_2(\Omega)} + 2\int_\Omega \mu_* |D(u)|^2_2\,dx
 -\int_{\Omega}{\rm div}\,(T_*\bar u)\,dx\\
&=\lambda |\varrho^{1/2}u|^2_{L_2(\Omega)}+ 2|\upmu_*^{1/2}D(u)|^2_{L_2(\Omega)} 
+\int_\Sigma([\![T_*\nu_\Sigma]\!]| u )\,dx \\
&=\lambda |\varrho^{1/2}u|^2_{L_2(\Omega)} + 2|\upmu_*^{1/2}D(u)|^2_{L_2(\Omega)} 
+\sigma\bar{\lambda} (\cA_\Sigma h|h)_{L_2(\Sigma)},
\end{aligned}
\end{equation}
where the relations 
$[\![u]\!]=0$, $[\![T_*\nu_\Sigma]\!]=\sigma \cA_\Sigma h\nu_\Sigma$ and 
$(u |\nu_\Sigma) = \lambda h$
are employed.
Taking the real part yields the identity
\begin{equation}
\begin{aligned}\label{evid-2}
0&={\rm Re}\,\lambda |\varrho^{1/2}u|^2_{L_2(\Omega)}+ 2|\upmu_*^{1/2}D(u)|^2_{L_2(\Omega)}
+\sigma {\rm Re}\,\lambda (\cA_\Sigma h|h)_{L_2(\Sigma)}.
\end{aligned}
\end{equation}
On the other hand, if ${\rm Im}\,\lambda\neq 0$, taking the imaginary part results in
$$ \sigma (\cA_\Sigma h|h)_{L_2(\Sigma)} =  |\varrho^{1/2}u|^2_{L_2(\Omega)}. $$
Substituting this expression into \eqref{evid-2} leads to
\begin{equation*}
0 = 2{\rm Re}\,\lambda |\varrho^{1/2}u|^2_{L_2(\Omega)} + 2|\upmu_*^{1/2}D(u)|^2_{L_2(\Omega)},
\end{equation*}
which shows that if ${\rm Re}\,\lambda \ge 0$ is an eigenvalue of $-L$, then it must be real.
In fact, the last identity shows that $D(u)=0$ and hence  $\nabla u=0$ by Korn's inequality.
Therefore, $u$ is constant on $\Omega$, and  hence  $u\equiv 0$ by the no-slip boundary condition on $\partial\Omega$.
The condition $\lambda h=(u|\nu_\Sigma)$ then yields $h=0$,
a contradiction to the assumption $(u,h)\neq (0,0)$.

In a next step it will be shown that \eqref{evid-1} does not have eigenvalues with  $\lambda>0$.
Assume to the contrary that \eqref{evid-1} has a nontrivial solution $(u,h)$ for $\lambda>0$.
Let $\Omega_{1,k}$ denote the components of $\Omega_1$ and set $\Sigma_k=\partial\Omega_{1,k}$. 
By the divergence theorem  
$$ 0 = \int_{\Omega_{1,k}} {\rm div}\,u\,dx = \int_{\Sigma_k} (u|\nu_\Sigma) \,d\Sigma_k
= \lambda\int_{\Sigma_k} h\,d\Sigma_k.$$
\goodbreak
This shows that the  mean values of $h$ vanish for all components of $\Sigma$.
As $\cA_\Sigma$ is positive semi-definite on functions which have mean value zero for each component of $\Sigma$, 
\eqref{evid-2} implies $\lambda=0$. Hence there are no eigenvalues with nonnegative real part. 

\medskip
Concerning the eigenvalue problem \eqref{ev-lin-theta}
one obtains, after taking the $L_2$-inner product of the first line with  $\vartheta$, 
integrating by parts, and emplying the condition $[\![d_*\partial_\nu\vartheta]\!]=0$,
the following relation
\begin{equation}
\label{evid-3}
0= \lambda|(\varrho\kappa_*)^{1/2}\vartheta|^2_{L_2(\Omega)} + |d_*^{1/2}\nabla\vartheta|^2_{L_2(\Omega)}.
\end{equation}
This readily shows that all eigenvalues of \eqref{ev-lin-theta} are real and non-positive.

\medskip
\noindent
%%%%%%%%%%%%%%%%%%%%%%%%%%%%%%%%%%%
(d)  Suppose $\lambda=0$. Then \eqref{evid-2} and \eqref{evid-3} yield 
$$|\upmu_*^{1/2}D(u)|^2_{L_2(\Omega)}=|d_*^{1/2}\nabla\vartheta|^2_{L_2(\Omega)}=0,$$
hence $\vartheta$ is constant and $D(u)=0$, and then $u=0$ by  Korn's inequality 
and the no-slip condition $u=0$ on $\partial\Omega$.
This implies further that the pressures are constant in the components of the phases. 
From the relation $[\![\pi]\!]\big|_{\Sigma_k}=-\sigma\cA_{\Sigma_k} h_k $
one concludes that $\cA_{\Sigma_k} h_k$ is constant for $k=1,\cdots,m$, where $h_k=h|_{\Sigma_k}$.

The kernel of the linearization $L$ is spanned by $e_\theta=(0,1,0)$, $e_{jk}=(0,0,Y_k^j)$, 
with $Y_k^j$ the spherical harmonics  of degree one for the spheres $\Sigma_k$, $j=1,\cdots,n$, $k=1,\cdots,m$, 
and $e_{0k}=(0,0,Y_k^0)$, where $Y_k^0$ equals one on $\Sigma_k$ and zero elsewhere.
Hence the dimension of the null space ${\sf N}(L)$ is $m(n+1)+1$. 
%Thus the dimension of the eigenspace for eigenvalue $\lambda=0$ is the same as the dimension of the 
%manifold $\cE$ of equilibria, namely $m(n+1)+1$ if $\Omega_1$ has $m$ components.

\medskip

Next it will be shown  that $\lambda=0$ is semi-simple. 
So suppose  $(u,\vartheta,h)$ is a solution of $L(u,\vartheta,h)= \sum_{j,k}\alpha_{jk}e_{jk} + \sum_k\beta_k e_{0k} +\gamma e_\theta$. 
This means
%%%%%%%%%
\begin{equation}
\begin{aligned}\label{simple-evu}
-\upmu_* \Delta u +\nabla\pi &= 0 && \mbox{in} && \Omega\setminus{\Sigma},\\
{\rm div}\, u &= 0 && \mbox{in} && \Omega\setminus{\Sigma},\\
u&=0  && \mbox{on} && \partial\Omega,\\
[\![u]\!] &= 0      && \mbox{on} && {\Sigma},\\
 -[\![T_*\nu_\Sigma]\!] +\sigma (\cA_\Sigma h)\nu_\Sigma &= 0 && \mbox{on} && {\Sigma},\\
-(u|\nu_\Sigma) &= \sum_{j,k}\alpha_{jk}Y_k^j+\sum_k\beta_k Y^0_k &&  \mbox{on} && {\Sigma},\\
\end{aligned}
\end{equation}
and
%%%%%%%%%%%
\begin{equation}
\begin{aligned}\label{simple-evtheta}
\phantom{}
\hspace{-5mm}-d_*\Delta \vartheta &=\varrho\kappa_*\gamma && 
\mbox{in} && \Omega\setminus{\Sigma},\\
\partial_\nu\vartheta&=0 && \mbox{on} && \partial\Omega,\\
[\![\vartheta]\!]&=0 && \mbox{on} && {\Sigma},\\
-[\![d_*\partial_\nu \vartheta]\!]&= 0 && \mbox{on} && {\Sigma}.\\
\end{aligned}
\end{equation}
It is to be shown that $(\alpha_{jk},\beta_k,\gamma)=0$ for all $j,k$.
Integrating the divergence equation for $u$ over $\Omega_{1,k}$ yields
$$
0=\int_{\Omega_{1,k}} {\rm div}\,u\,dx = \int_{\Sigma_k} (u|\nu_\Sigma)\,d\Sigma_k =
- \sum_k \beta_k  \int_{\Sigma_k} Y_k^0 \,d\Sigma_k = -\beta_k |\Sigma_k|,$$
where the property that the spherical harmonics have mean value zero is employed.
Therefore, $\beta_k=0$ for $k=1,\cdots,m$.

Taking the $L_2$-inner product of the equation for $u$ with $u$, one obtains as in \eqref{evid-1}
$$|\upmu_*^{1/2}D(u)|^2_{L_2(\Omega)}=0.$$
This implies $D(u)=0$, hence $u=0$ by Korn's inequality and the no-slip boundary condition on $\partial\Omega$.
This, in turn, yields
$$ 0= -(u|\nu_\Sigma) = \sum_{j,k}\alpha_{jk}Y_k^j.$$
Thus $\alpha_{jk}=0$ for all $j,k$, as the spherical harmonics $Y_k^j$ are linearly independent.
Finally, integrating the equation for $\vartheta$ yields
$$
\gamma (\kappa_*|\varrho)_{L_2(\Omega)}=  - \int_\Omega d_*\Delta \vartheta \,dx 
=  \int_\Sigma [\![d_*\partial_\nu \vartheta]\!] \,d\Sigma 
-\int_{\partial\Omega} d_*\partial_\nu \vartheta\,d(\partial\Omega) =0,
$$
and hence $\gamma =0$. Therefore, the eigenvalue $\lambda=0$ is semi-simple.

\medskip
\noindent
(e) The assertion follows from  the fact that $T_{e_*}\cE\subset {\sf N}(L)$ and the relation
$${\rm dim}\,{\sf N}(L)={\rm dim}\,T_{e_*}\cE=m(n+1)+1.$$
\end{proof}
%%%%%%%%%%%%%%%%%%%%%%%%%%%%%%%%%%%%%%%%%%%%
\section{Nonlinear stability of equilibria}
%%%%%%%%%%%%%%%%%%%%%%%%%%%%%%%%%%%%%%%%%%%%
Suppose  $e_*=(0,\theta_*,\Sigma)\in\cE$ is a non-degenerate fixed equilibrium.
Choosing $\Sigma$ as reference manifold and employing the Hanzawa transformation
to problem \eqref{NS-heat}-\eqref{kinematic} one obtains the nonlinear system
\begin{equation}
\begin{aligned}\label{nonlin-u}
\varrho\partial_t u-\upmu_* \Delta u +\nabla\pi &=F_u(u,\vartheta,h,\pi) && \mbox{in} && \Omega\setminus{\Sigma},\\
{\rm div}\, u &=G_d(u,h) && \mbox{in} && \Omega\setminus{\Sigma},\\
u&=0\quad && \mbox{on} && \partial\Omega,\\
[\![u]\!] &= 0      && \mbox{on} && {\Sigma},\\
 -\cP_\Sigma[\![2\upmu_*D(u)\nu_\Sigma ]\!] &=G_\tau (u,\vartheta,h) && \mbox{on} && {\Sigma},\\
 -([\![2\upmu_* D(u)\nu_\Sigma)]\!]\,|\nu_\Sigma)+[\![\pi ]\!]+\sigma \cA_\Sigma h &=G_\nu(u,\vartheta,h) && \mbox{on} && {\Sigma},\\
u(0)&=u_0 &&\mbox{in} && \Omega,
\end{aligned}
\end{equation}

\begin{equation}
\begin{aligned}\label{nonlin-theta}
\hspace{2cm}
\varrho\kappa_*\partial_t \vartheta -d_*\Delta \vartheta &= F_\theta(u,\vartheta,h)
 && \mbox{in} && \Omega\setminus{\Sigma},\\
\partial_\nu\vartheta&=0 && \mbox{on} && \partial\Omega,\\
[\![\vartheta]\!]&=0 && \mbox{on} && {\Sigma},\\
-[\![d_*\partial_\nu \vartheta]\!]&= G_\theta(\vartheta,h) && \mbox{on} && {\Sigma},\\
\vartheta(0)&=\vartheta_0  && \mbox{in} && \Omega\setminus\Sigma,
\end{aligned}
\end{equation}
and
\begin{equation}
\begin{aligned}\label{nonlin-h}
\hspace{1.5cm}
\partial_t h-(u|\nu_\Sigma) &= F_h(u,h)
&&\mbox{on} && \Sigma,\\
h(0)&=h_0 &&\mbox{on} && \Sigma, \\
\end{aligned}
\end{equation}
%%%%%%%
where $\vartheta$,
$\kappa_*$, $\upmu_*$, $d_*$, and 
$\cA_\Sigma$ have the same meaning as in Section 3.
The precise expressions for the nonlinearities will not be listed here,
and the reader is referred to \cite[Chapter 9]{PrSi16}
(and also to \cite{KPW13} for the isothermal case).
It suffices to point out that the nonlinearities are $C^1$ in all variables
and vanish, together with their first order derivatives, at $(u,\vartheta,h,\pi)=(0,0,0,c)$,
where $c$ is constant in the phase components.

To obtain an abstract formulation of problem \eqref{nonlin-u}-\eqref{nonlin-h}, 
we choose as the principal system variable $z=(u,\vartheta,h)$.
The regularity space for $z$ is
\begin{equation*}
\begin{aligned}
\EE(a) := \! \big\{(u,\vartheta,h)\in\EE_{u}(a)\times \EE_{\theta}(a)\times \EE_{h}(a)\!: 
  u,\partial_\nu\vartheta =0 \;\mbox{on}\;\partial\Omega,\;
 [\![u]\!], [\![\vartheta]\!]=0 \;\mbox{on}\; \Sigma\big\},
\end{aligned}
\end{equation*}
where 
\begin{align*}
& \EE_{u}(a)=\EE_{\theta}(a)^n, \quad \EE_{\theta}(a)= H^1_{p}(J;L_p(\Omega))\cap L_{p}(J;H^2_p(\Omega\setminus\Sigma)),\\
&\EE_{h}(a) = W^{2-1/2p}_{p}(J;L_p(\Sigma))\cap H^1_{p}(J;W^{2-1/p}_p(\Sigma))\cap L_{p}(J;W^{3-1/p}_p(\Sigma))
\end{align*}
and $J=(0,a)$.
Here the regularity for the height function $h$ 
is accounted for as follows.
Asserting that %the velocity field $u$ satisfies  
$u\in H^1_{p}(J;L_p(\Omega))\cap L_{p}(J;H^2_p(\Omega\setminus\Sigma))$ one obtains by trace theory
\begin{equation*}
\begin{aligned}
&([\![2\upmu_* D(u)\nu_\Sigma)]\!]\,|\nu_\Sigma)\in  W^{1/2-1/2p}_p(J;L_p(\Sigma))\cap L_p(J;W^{1-1/p}_p(\Sigma)),\\
&(u|\nu_\Sigma)\in W^{1-1/2p}_p(J;L_p(\Sigma))\cap L_p(J;W^{2-1/p}_p(\Sigma)).
\end{aligned}
\end{equation*}
Requiring that the function $h$ in \eqref{nonlin-u}, \eqref{nonlin-h} has the best possible regularity then amounts to
\begin{equation*}
\begin{aligned}
&\Delta_\Sigma h \in  W^{1/2-1/2p}_p(J;L_p(\Sigma))\cap L_p(J;W^{1-1/p}_p(\Sigma)), \\
&\partial_t h\in W^{1-1/2p}_p(J;L_p(\Sigma))\cap L_p(J;W^{2-1/p}_p(\Sigma)),
\end{aligned}
\end{equation*}
which in turn results in $h\in \EE_h(a)$, as $\EE_h(a)$ embeds into $W^{1/2-1/2p}_p(J;H^2_p(\Sigma))$.   
By the same reasoning one also has $[\![\pi]\!]\in W^{1/2-1/2p}_p(J;L_p(\Sigma))\cap L_p(J;W^{1-1/p}_p(\Sigma))$. 
\\

The trace space  $X_\gamma$  of $\EE(a)$ is given by
\begin{equation*}
\begin{aligned}
 X_\gamma &=\big\{ (u,\vartheta,h)\in W^{2-2/p}_p(\Omega\setminus\Sigma)^{n+1}\times  W^{3-2/p}_p(\Sigma): 
           u, \partial_\nu\vartheta=0 \mbox{ on }\partial\Omega \\
        &\hspace{7.5cm} \;[\![u]\!],[\![\vartheta]\!]=0 \mbox{ on } \Sigma\big\}.
\end{aligned}
\end{equation*}
Finally, let
\begin{equation*}
\label{hat-EE,hat-X}
\begin{aligned}
&\hat \EE(a)   =\{w=(z,\pi): z\in\EE(a),\; \pi\in L_p(J;\dot H^1_p(\Omega\setminus\Sigma))\}, \\
&\hat X_\gamma =\{w=(z,\pi): z\in X_\gamma,\; \pi\in \dot W^{1-2/p}_p(\Omega\setminus\Sigma),\ [\![ \pi]\!]\in W^{1-3/p}_p(\Sigma)\}, 
\end{aligned}
\end{equation*}
where $\dot H^1_p(\Omega\setminus\Sigma)$ and $\dot W^{1-2/p}_p(\Omega\setminus\Sigma )$ denote corresponding homogeneous spaces.

% $\dot H^1_p(\Omega\setminus\Sigma):=\{\phi\in L_{1,{\rm loc}}(\Omega\setminus\Sigma): 
% \nabla \phi\in L_p(\Omega\setminus\Sigma)\}$.
%%%%%%%%%%%%%%%%
\subsection{The tangent space at equilibria}
\noindent
In this subsection, the structure of the state manifold $\cSM$ in a neighborhood of a fixed equilibrium $e_*=(0,\theta_*,\Sigma)$
will be studied. Towards this objective, observe that
near $e_*$, the state manifold is described by
\begin{align*}
\cSM_*&=\big\{ (u,\vartheta,h)\in X_\gamma:\,  {\rm div}\, u= G_d(u,h) \mbox{ in } \Omega\setminus\Sigma,\\
       & \qquad -\cP_\Sigma[\![2\upmu_*D(u)\nu_\Sigma]\!]= G_\tau(u,\vartheta,h),
       \; -[\![d_*\partial_\nu\vartheta]\!]= G_\theta(\vartheta,h) \mbox{ on } \Sigma\big\}.
\end{align*}
Associated to $\cSM_*$ is the linear subspace
\begin{align*}
\cSX_*&=\big\{ (u,\vartheta,h)\in X_\gamma:\,  {\rm div}\, u=0\mbox{ in } \Omega\setminus\Sigma,\\
        &\qquad \cP_\Sigma[\![\upmu_*D(u)\nu_\Sigma]\!]= 0,\, 
        [\![d_*\partial_\nu\vartheta]\!]= 0 \mbox{ on } \Sigma \big\};
\end{align*}
the boundary trace space
$$Y_\gamma= W^{1-2/p}_{p,0}(\Omega\setminus\Sigma)\times W^{1-3/p}_p(\Sigma;T\Sigma)\times W^{1-3/p}_p(\Sigma),$$ 
where
$$ W^{1-2/p}_{p,0}(\Omega\setminus\Sigma):= 
\{ v\in W^{1-2/p}_p(\Omega\setminus\Gamma):\, \int_{\Omega \setminus\Sigma} v\,dx =0\},$$
with $T\Sigma$ the tangent bundle of $\Sigma$;
the linear stationary boundary operator
$${\sf B}  z = ({\rm div}\, u, -\cP_\Sigma[\![2\upmu_* D(u)\nu_\Sigma]\!], 
-[\![d_*\partial_\nu\vartheta]\!]),$$
and the stationary boundary nonlinearity
$$ {\sf G }(z) =(G_d(u,h), G_\tau(u,\vartheta,h), G_\theta(\vartheta, h)).$$
With this notation one has
\begin{equation}
\begin{aligned}\label{AB}
\cSM_*  & = \{ z\in X_\gamma:\, {\sf B} z ={\sf G}(z) \mbox{ in } Y_\gamma\}, \\
 \cSX_* & = \{ z\in X_\gamma:\, {\sf B} z =0 \mbox{ in } Y_\gamma\}.
\end{aligned}
\end{equation}
This structure will now be employed to parameterize $\cSM_*$ over $\cSX_*$ near $(0,0,0)$. 
This shows, in particular, that $\cSX_*$ is isomorphic to the tangent space of $\cSM_*$ at $(0,0,0)$, or equivalently, to the tangent space of $\cSM$ at $e_*$. 

It will be convenient to enlarge the system variable $z$ by the pressure, 
setting $w=(z,\pi)$, where $\pi \in \dot{W}^{1-2/p}_p(\Omega\setminus\Sigma)$ and $[\![\pi]\!]\in W^{1-3/p}_p(\Sigma)$.
Furthermore, including the normal component of the normal stress balance
$$ 
-([\![2\upmu_* D(u)\nu_\Sigma)]\!]\,|\nu_\Sigma) + [\![\pi]\!] +\sigma \cA_\Sigma h =G_\nu(u,\vartheta,h)
$$
leads to the extended operators  $\hat{\sf B}$ and $\hat{\sf G}$. 
Is is worthwhile to point out that here the pressure appears only linearly,
i.e., it does not appear in $\hat{\sf G}$. 

The differential operator ${\sf A}$ is defined by the expression for the operator $L$ 
introduced in Section~3. 
Consequently, system \eqref{nonlin-u}-\eqref{nonlin-h} can be restated as
\begin{equation}
\label{full-problem}
\begin{aligned}
\partial_t z +{\sf A}w & = \hat{\sf F}(w)          && \mbox{in} && \Omega\setminus\Sigma, \\
         \hat{\sf B} w & = \hat{\sf G}(z)  && \mbox{on} && \Sigma, \\
                  z(0) & = z_0, && &&
\end{aligned}
\end{equation}
where  $\hat{\sf F}(w):=(F_u(w)/\varrho,F_\theta(z)/\varrho \kappa_*,F_h(z))$.
It is important to note that, formally, $(\hat{\sf F}(0),\hat{\sf F}^\prime(0))=0$ and $(\hat {\sf G}(0),\hat {\sf G}^\prime (0))=0$, with
$\hat{\sf F}^\prime$ and $\hat {\sf G}^\prime$ the Fr\'echet derivative of $\hat{\sf F}$ and $\hat {\sf G}$, respectively.
%%%%%%%%%%%%%%%%%%
\subsection{Parameterization of $\cSM$.}
%%%%%%%%%%%%%%%%%

In order to parameterize $\cSM_*$ over $\cSX_*$ one solves the problem
\begin{equation}
\begin{aligned}\label{param-state-mf-nl}
\omega \overline{z} + {\sf A}_0\overline{w} &=0 &&\mbox{in}  &&\Omega\setminus\Sigma,\\
\hat{{\sf B}} \overline{w} &= \hat{\sf G}(\overline{z}+\tilde{z})  && \mbox{on} && \Sigma,
\end{aligned}
\end{equation}
where $\omega>0$ is sufficiently large and 
$$ {\sf A}_0w:= (-(\upmu_*/\varrho)\Delta u + (1/\varrho)\nabla\pi , -(d_*/\varrho\kappa_*)\Delta\vartheta,0),
\quad w=(u,\vartheta,h,\pi)\in \hat X_\gamma .$$
 Given $\tilde{z}\in \cSX_*$ small, 
we are looking for a solution $\overline{w}\in \hat{X}_\gamma$. 
For this, the implicit function theorem will be employed. 
Obviously, for $\tilde{z}=0$ one has the trivial solution $\overline{w}=0$. 
Moreover, one notes that the first line in \eqref{param-state-mf-nl} yields $\overline{h}=0$.
As $\hat{\sf G}:X_\gamma\to \hat{Y}_\gamma$ is of class $C^1$ with $(\hat{\sf G}(0), \hat{\sf G}^\prime(0))=0$, 
where $\hat Y_\gamma = Y_\gamma\times W^{1-3/p}_p(\Sigma)$, it is to be shown that the linear problem
\begin{equation}
\label{param-state-mf}
\begin{aligned}
\omega \overline{z} + {\sf A}_0 \overline{w} &=0 && \mbox{in} &&\Omega\setminus\Sigma,\\
\hat{{\sf B}} \overline{w} &= \hat{g}   && \mbox{on} && \Sigma,
\end{aligned}
\end{equation}
admits a unique solution, for any given datum $\hat{g} \in \hat{Y}_\gamma$. 
In fact, the propositions in the next subsection will do this job, up to lower order perturbations.
Therefore, one may apply the implicit function theorem to find a ball $B_{\cSX_*}(0,r)$ and a map 
$$\hat\phi: B_{\cSX_*}(0,r)  \to \hat{X}_\gamma$$ 
of class $C^1$ with $(\hat\phi(0),\hat\phi^\prime(0))=0$ 
such that $\overline{w}=\hat\phi(\tilde{z})$ is the unique solution of \eqref{param-state-mf-nl} near zero. 
The map 
$({\rm id} +\phi):B_{\cSX_*}(0,r)\to \cSM_*$ is surjective onto a neighborhood of zero,
where $\phi$ means dropping the pressure $\pi$ in $\hat\phi$. 
To see this fix any $z\in \cSM_*$ and solve the linear problem 
with  \eqref{param-state-mf} $ \hat g=\hat{\sf G}(z)$
to obtain a unique $\overline{w}=(\overline{z},\overline{\pi})\in \hat{X}_\gamma$. 
Let then $\tilde{z}= z-\overline{z}$. If $z$ is chosen small enough, 
$\tilde{z}\in B_{\cSX_*}(0,r)$ and hence $\overline{w}=\hat\phi(\tilde{z})$ by uniqueness,
yielding the representation $z=\tilde z +\phi(\tilde z)$.
Consequently,
the map $\Phi:B_{\cSX_*}(0,r)\to \cSM_*$ defined by 
\begin{equation}
\label{Phi-SM}
\Phi(\tilde{z}) = \tilde{z} + \phi(\tilde{z})
\end{equation}
%where $\phi$ means dropping the pressure $\pi$ in $\hat\phi$, 
yields the desired parameterization. In conclusion, the following result has been obtained.
%%%%%%%%%%%%%%%%%
\begin{theorem}\label{tangent-space-SM}  
The state manifold $\cSM_*$ can be parameterized via the map $\Phi$ over the space $\cSX_*$. 
In particular, the tangent space $T_{e_*}\cSM$ at the equilibrium $e_*$, and equivalently  the tangent space $T_0\cSM_*$ at zero, 
is isomorphic to the space $\cSX_*$.
\end{theorem}

\noindent
Note that an equilibrium $e_\infty\in \cE$ close to $e_*\in\cE$ in $\cSM$, 
respectively $z_\infty$ close to zero in $X_\gamma$, decomposes as
$$ z_\infty = \tilde{z}_\infty +\overline{z}_\infty = \tilde{z}_\infty +\phi(\tilde{z}_\infty),$$
with $\tilde{z}_\infty \in \cSX_*$. This follows as   $ {\sf A} w_\infty =\hat{\sf F}(w_\infty)=0$ at an equilibrium.
%%%%%%%%%%%%%%%
\subsection{Auxiliary linear elliptic problems}
For the application of the implicit function theorem in Section 4.2 the following results were employed.
The first one concerns an elliptic transmission problem for the temperature,
and the second one a two-phase Stokes problem.

\begin{proposition}
Let $\omega>0$ be large, $\varrho,\kappa_*,d_*>0$,  and $ p> n+2$. Then the problem
\begin{equation*}
\begin{aligned}
\varrho\kappa_*\omega \vartheta - d_* \Delta \vartheta &=0 && \rm{in} && \Omega\setminus \Sigma,\\
\partial_\nu \vartheta &= 0  &&  \rm{on} && \partial\Omega,\\
[\![\vartheta]\!]=0, \quad -[\![d_*\partial_\nu \vartheta]\!]&=g &&  \rm{on} && \Sigma,\\
\end{aligned}
\end{equation*}
has a unique solution $\vartheta\in W^{2-2/p}_p(\Omega\setminus\Sigma)$ if and only if $g\in W^{1-3/p}_p(\Sigma)$.
\end{proposition}
\noindent
The assertion follows, for instance, from the results in \cite[Section 6.5]{PrSi16}.
%
%
%\medskip
%\noindent
%The next proposition concerns the linear Stokes problem.

\begin{proposition}
Let $\omega>0$ be large, $\varrho,\upmu_*>0$, and $p>n+2$. Then the problem
\begin{equation*}
\begin{aligned}
\varrho \omega u -\upmu_* \Delta u +\nabla \pi &=0 && \rm{in} && \Omega\setminus\Sigma,\\
{\rm div} \, u &= g_d  && \rm{in} && \Omega\setminus\Sigma,\\
u&=0  &&  \rm{on} && \partial\Omega,\\
 [\![u]\!]=0,\quad
    - \cP_\Sigma [\![2\upmu_*D(u)\nu_\Sigma]\!] &= g_\tau && \rm{on} && \Sigma,\\
([\![2\upmu_* D(u)\nu_\Sigma ]\!]\,|\nu_\Sigma ) +[\![\pi]\!]  &= g_\nu  && \rm{on} && \Sigma,\\
\end{aligned}
\end{equation*}
has a unique solution
$$u\in W^{2-2/p}_p(\Omega\setminus\Sigma), \;\; \pi\in \dot{W}^{1-2/p}_p(\Omega\setminus\Sigma),
\;\; [\![\pi]\!]\in W^{1-3/p}_p(\Sigma), $$
 if and only if
$$ g_d\in W^{1-2/p}_{p,0}(\Omega\setminus\Sigma),\quad (g_\tau,g_\nu)\in W^{1-3/p}_p(\Sigma;T\Sigma\times\R).$$
\end{proposition}
\noindent
The assertion follows, for instance, from \cite[Chapter 8]{PrSi16}.
%%%%%%%%%%%%%%%%%
\subsection{Nonlinear stability analysis}
%%%%%%%%%%%%%%%%%%
In oder to analyze the stability properties of an equilibrium $e_*=(0,\theta_*,\Sigma)$,
the time-dependent variables are decomposed in the same way as in the previous section into 
$z(t)=\overline{z}(t) +\tilde{z}(t)$
and $w(t)=\overline{w}(t)+\tilde{w}(t)$.
The full problem \eqref{full-problem} may then be decomposed 
into two systems, formally one for $\overline{w}$ and one for $\tilde{w}$, according to
\begin{equation}
\begin{aligned}\label{bar-problem}
(\omega+\partial_t) \overline{z} + {\sf A} \overline{w} &=\hat{\sf F}(\overline{w} +\tilde{w}) && \mbox{in} &&\Omega\setminus\Sigma,\\
\hat{\sf B} \overline{w} &= \hat{{\sf G}}(\overline{z}+\tilde{z}) && \mbox{on} && \Sigma,\\
\overline{z}(0)&=\phi(\tilde{z}_0) && \mbox{in} && \Omega,
\end{aligned}
\end{equation}
and 
\begin{equation}
\begin{aligned}\label{tilde-problem}
\partial_t\tilde{z} + {\sf A} \tilde{w} &= \omega \overline{z} && \mbox{in} &&\Omega\setminus\Sigma,\\
\hat{\sf B} \tilde{w} &= 0  && \mbox{on} && \Sigma,\\
\tilde{z}(0)&=\tilde{z}_0 && \mbox{in} && \Omega.
\end{aligned}
\end{equation}
Adding these equations yields problem \eqref{full-problem}.

\medskip
One should think of this decomposition in the following way. The first part has a fast dynamics due to $\omega>0$ large and takes care of the stationary boundary conditions, while the second equation lives in the tangent space $\cSX_*$  and carries the actual dynamics. 

There should be a word of caution. 
While for the initial value $z_0$ as well as for $z_\infty$  
the decomposition $z_0=\tilde{z}_0 +\phi(\tilde{z}_0)$
and $z_\infty=\tilde{z}_\infty +\phi(\tilde{z}_\infty)$
is employed, it no longer holds in the time-dependent case; in general $\overline{z}(t)\neq \phi(\tilde{z}(t))$!

In order to show stability and exponential convergence of solutions starting close to 
the equilibrium $e_*=(0,\theta_*,\Sigma)$, the decomposition 
$z= \overline{z}+\tilde{z} +z_\infty$ is used, 
with the idea that $z_\infty$ will be the limit of $z(t)$ as $t$ goes to infinity, 
and $\overline{z}, \tilde{z}$ are exponentially decaying. This means that the corresponding equations for $\overline{z}$ and $\tilde{z}$ are shifted to
\begin{equation}
\begin{aligned}\label{bar-shifted}
(\omega+\partial_t) \overline{z} + {\sf A} \overline{w} &=\hat{\sf F}(\overline{w} +\tilde{w}+w_\infty)- \hat{\sf F}(w_\infty) 
&& \mbox{in} &&\Omega\setminus\Sigma,\\
\hat{\sf B} \overline{w} &= \hat{{\sf G}}(\overline{z}+\tilde{z}+z_\infty)-\hat{\sf G}(z_\infty) && \mbox{on} && \Sigma,\\
\overline{z}(0)&=\phi(\tilde{z}_0)-\phi(\tilde{z}_\infty) && \mbox{in} && \Omega,
\end{aligned}
\end{equation}
and
\begin{equation}
\begin{aligned}\label{tilde-shifted}
\hspace{-1cm}
\partial_t\tilde{z} + {\sf A} \tilde{w} &= \omega \overline{z} && \mbox{in} && \Omega\setminus\Sigma,\\
\hat{\sf B} \tilde{w} &= 0  && \mbox{on} && \Sigma,\\
\tilde{z}(0)&=\tilde{z}_0-\tilde{z}_\infty && \mbox{in} && \Omega,
\end{aligned}
\end{equation}
where $z_0=\tilde{z}_0 +\phi(\tilde{z}_0)$ and $z_\infty=\tilde{z}_\infty +\phi(\tilde{z}_\infty)$. 
It is convenient to remove the pressure $\tilde{\pi}$ from \eqref{tilde-shifted} by solving the weak transmission problem 
\begin{equation*}
\begin{aligned}
(\varrho^{-1}\nabla \tilde\pi|\nabla\phi )_{L_2(\Omega)} 
&=(\varrho^{-1}\upmu_*\Delta \tilde u|\nabla \phi)_{L_2(\Omega)} +\omega (\overline{u}|\nabla\phi)_{L_2(\Omega)} ,
&&\phi\in {H}^1_{p^\prime}(\Omega),\\
[\![\tilde\pi]\!] &=-\sigma \cA_\Sigma \tilde h 
+ ([\![2\upmu_* D(\tilde u)\nu_\Sigma ]\!]\,|\nu_\Sigma)&&\mbox{on}\;\; \Sigma
\end{aligned}
\end{equation*}
and insert it into \eqref{bar-shifted} for $\tilde w$.
The first problem may be written abstractly as
\begin{equation}
\begin{aligned} \label{bar-abstract-infty}
\LL_\omega \overline{w} =N(\overline{w}, \tilde{z}, \tilde{z}_\infty),\quad t>0, 
\quad \overline{z}(0)=\phi(\tilde{z}_0)-\phi(\tilde{z}_\infty),
\end{aligned}
\end{equation}
and with the Helmholtz projection $\PP$, the second one as the evolution equation
\begin{align} \label{tilde-abstract-infty}
\partial_t \tilde{z} + L \tilde{z} = \omega{\mathbb P} \overline{z}, \quad t>0,\quad \tilde{z}(0)=\tilde{z}_0- \tilde{z}_\infty.
\end{align}
Here $L:X_1\to X_0$ is the operator defined in Section 3. 
For further use 
the space $\FF(a)$ of data $(f_u,f_\theta,f_h,g_d,g_u,g_\theta)$,
$$ \FF(a)= \FF_1(a)\times\FF_2(a)\times\FF_3(a)\times\FF_4(a)\times \FF_5(a)\times\FF_6(a),$$
is introduced, with 
\begin{equation*}
\begin{aligned}
& \FF_1(a)=L_p(J;L_p(\Omega))^n, \\
& \FF_2(a)=L_p(J;L_p(\Omega)), \\  
& \FF_3(a) = W^{1-1/2p}_p(J;L_p(\Sigma))\cap L_p(J;W^{2-1/p}_p(\Sigma)),\\
& \FF_4(a) = H^1_p(J;{_0\dot{H}}^{-1}_p(\Omega))\cap L_p(J;H^1_p(\Omega\setminus\Sigma)), \\
& \FF_5(a) = W^{1/2-1/2p}_p(J;L_p(\Sigma))^n\cap L_p(J;W^{1-1/p}_p(\Sigma))^n, \\
& \FF_6(a) = W^{1/2-1/2p}_p(J;L_p(\Sigma))\cap L_p(J;W^{1-1/p}_p(\Sigma)), \\
\end{aligned}
\end{equation*}
where $J=(0,a)$ and  
${_0\dot{H}}^{-1}_p(\Omega))=[\dot H^1_{p^\prime,\partial\Omega}(\Omega)]^*$ is the dual of the 
homogeneous space
$$\dot H^1_{p^\prime,\partial\Omega}(\Omega)
:=\{\phi\in L_{1,{\rm loc}}(\Omega):\nabla\phi\in L_{p^\prime}(\Omega),\; \phi=0\;\;\hbox{on}\;\;\partial\Omega\}.$$
In addition, the function space
\begin{equation}
\label{EE-tilde}
\tilde\EE(a):= H^1_{p}(J;X_0)\cap L_{p}(J;X_1),\quad J=(0,a),
\end{equation}
will be used, with $X_0$ and $X_1$ as in Section 3.
It follows from the results in Chapters 6 and 8 of \cite{PrSi16} that
$$  
(\LL_\omega,{\rm tr})\in{\rm Isom}(\hat{\EE}(\infty,\delta), \FF(\infty,\delta)\times X_\gamma),
$$ 
provided $\omega$ is chosen sufficiently large, that is,
$(\LL_\omega,{\rm tr})$ is an isomorphism 
from $\hat{\EE}(\infty,\delta)$ onto $\FF(\infty,\delta)\times X_\gamma$.
Here the following notation is employed:
$$ z\in \EE(\infty,\delta) \quad \Leftrightarrow \quad e^{\delta t} z \in \EE(\infty),$$
and similarly for $\FF(\infty,\delta)$, $\hat{\EE}(\infty,\delta)$ and $\tilde\EE(\infty,\delta)$.

According to Theorem~\ref{linstab-heat},  
$-L$ is the generator of an analytic $C_0$-semigroup with maximal $L_p$-regularity in $X_0$. 
Moreover, it follows from the same theorem that
there is a number $\delta_0>0$ such that
${\rm Re}\,\sigma(-L)\cap  (-\delta_0,0)=\emptyset$. Let $\delta$ be chosen so that $0<\delta<\delta_0$.
On shows that the function 
$N$ is of class $C^1$ with respect to the variables 
$(\overline{w},\tilde{z})$ in the function spaces  
$\hat\EE(\infty,\delta)\times\tilde\EE(\infty,\delta)$,
provided  condition \eqref{condition-H} holds, 
but merely continuous in $\tilde{z}_\infty$, 
unless one additional degree of regularity is imposed on the coefficients.
\medskip

\noindent
The main theorem of this chapter is the following.
%%%%%%%%%%%%%%%%%%%%%%%%
\begin{theorem} \label{nonlinear-stability}
 Let $p>n+2$ and suppose that condition \eqref{condition-H} holds.

Then each equilibrium $e_*=(0,\theta_*,\Sigma)\in\cE$
 is nonlinearly stable in the state manifold $\cSM$. Any solution with initial value close to $e_*$ in $\cSM$ exists globally
and converges in $\cSM$ to a possibly different stable equilibrium $e_\infty\in\cE$ at an exponential rate.
\end{theorem}
%%%%%%%%%%%%%%%
\begin{proof}
Based on the spectral properties of $L$ derived in Theorem~\ref{linstab-heat},
let $P^s$ denote the projection onto the stable subspace $X_0^s=P^s X_0={\sf R}(L)$ 
and $P^c$ the complementary projection onto the kernel ${\sf N}(L)=X_0^c=P^c X_0$.
Moreover, let $L^s$ be the part of $L$ in $X_0^s$.

Let ${\sf y} = P^s\tilde{z}$, ${\sf x}=P^c\tilde{z}$ and note that 
the equilibria over $P^c\cSX_*=P^cX_0$ may be parameterized according to
$$ z_\infty = {\sf x}_\infty +\psi({\sf x}_\infty) +\phi({\sf x}_\infty +\psi({\sf x}_\infty)),
\quad {\sf x}_\infty \in X^c_0,$$
by solving the nonlinear stationary problem
$$L^s y =\omega P^s\PP\phi({\sf x}+{\sf y})$$
by the implicit function theorem. Finally, let ${\sf y}_\infty = \psi({\sf x}_\infty)$.
Applying the projection $P^s$ to the equation for $\tilde{z}$ one obtains the problem
$$\partial_t {\sf y} + L^s {\sf y} =\omega P^s\PP\,\overline{z},\quad t>0,\quad {\sf y}(0)= {\sf y}_0- {\sf y}_\infty,$$
and for ${\sf x}$ analogously
$$ \partial_t {\sf x} = \omega P^c\PP\,\overline{z},\quad t>0, \quad {\sf x}(0) = {\sf x}_0-{\sf x}_\infty.$$
It is important to observe that
\begin{equation}
\label{L-stable}
(\partial_t + L^s,{\sf tr})\in {\rm Isom}(P^s\tilde\EE(\infty, \delta),L_p(\R_+,\delta; X^s_0)\times X^s_\gamma).
\end{equation}
Finally, the whole problem \eqref{bar-abstract-infty}-\eqref{tilde-abstract-infty} 
may be rewritten as $\HH({\sf v},({\sf x}_\infty,{\sf y}_0))=0$, 
where ${\sf v}= ( \overline{w}, {\sf y},{\sf x},{\sf x_0})$ and
\begin{align*}
\mbox{}\HH({\sf v}, ({\sf x}_\infty,{\sf y}_0))=
\left[
\begin{array}{c}
\big(\LL_\omega\overline{w}- N({\sf v}, {\sf x}_\infty), 
\overline{z}(0)-\phi({\sf x}_0+{\sf y}_0)+ \phi({\sf x}_\infty+{\sf y}_\infty)\big)\\[0.4em]
\big(\partial_t {\sf y} + L^s{\sf y} - \omega P^s\PP\,\overline{z}, 
{\sf y}(0)- {\sf y}_0 + {\sf y}_\infty \big)\\[0.4em]
{\sf x}(t) + \omega\int_t^\infty P^c\PP\,\overline{z}(s)\,ds\\[0.4em]
{\sf x}_0 -{\sf x}_\infty +  \omega\int_0^\infty P^c\PP\,\overline{z}(s)\,ds
\end{array} \right].
\end{align*}
One shows that the mapping
\begin{equation*}
\begin{aligned}
&\HH:\hat\EE(\infty,\delta)\times P^s\tilde\EE(\infty,\delta)\times P^c\tilde\EE(\infty,\delta)\times X_0^c
\times (X_0^c\times X^s_\gamma) \\
&\hspace{2cm}
\to (\FF(\infty,\delta)\times  X_\gamma )\times (L_p(\R_+;\delta, X^s_0)\times X^s_\gamma)
\times L_p(\R_+;\delta,X^c_0)\times X^c_0
\end{aligned}
\end{equation*}
is of class $C^1$ w.r.t $({\sf v}, {\sf y}_0)$, continuous w.r.t.~${\sf x}_\infty$,
and differentiable w.r.t.~${\sf x}_\infty $ at ${\sf x}_\infty=0$.
One notes that $X_0^c=X_\gamma^c=X_1^c$.
The Fr\'echet derivative $D_{\sf v}\HH(0,0)$ w.r.t.\ the variable ${\sf v}$ is given by the operator matrix
\begin{align}
\label{HH-derivative}
D_{\sf v}\HH(0,0) = \left[\begin{array}{cccc}
(\LL_\omega,{\sf tr}) & 0&0&0\\
*& (\partial_t+L^s,{\sf tr})&0&0\\
*&0& I&0\\
*&0&0&I
\end{array}\right].
\end{align}
Here the stars indicate bounded linear operators which, 
due to the triangular structure of the operator matrix,
do not need to be computed explicitly, as the diagonal terms of this operator matrix are invertible.
Therefore, by the implicit function theorem, see for instance \cite[Theorem 15.1]{Dei85},
 there are balls  $B_{X_0^c}(0,r)$ and $B_{X_\gamma^s}(0,r)$, and a continuous map 
$$\cT: B_{X_0^c}(0,r)\times B_{X_\gamma^s}(0,r)\to \hat{\EE}(\infty,\delta)\times \tilde{\EE}(\infty,\delta)\times X_0^c,
\quad \cT({\sf x}_\infty,{\sf y}_0) = (\overline{w},\tilde{z}, {\sf x}_0),$$
with $\cT(0,0)=0$. 
% $\HH(\cT({\sf x}_\infty,{\sf y}_0),({\sf x}_\infty,{\sf y}_0))=0.$ 
% Here it is reminded that $\overline{w}=(\overline{z},\overline{\pi})$.
Then $(z,\pi):=(\overline{z} +\tilde z + z_\infty, \overline\pi +\tilde{\pi}+\pi_\infty)$
yields the unique solution of \eqref{full-problem} such that
\begin{equation}
\label{z-convergence} 
 z(t)\to z_\infty := {\sf x}_\infty + \psi({\sf x}_\infty) + \phi({\sf x}_\infty + \psi({\sf x}_\infty))
\quad \mbox{in  $X_\gamma$ \ \ as $ t\to\infty$}.
\end{equation}
One should observe that $\tilde z:={\sf x} +{\sf y} \in\tilde\EE(\infty,\delta)$ seemingly has 
less regularity than its counterpart $\overline{z}\in \EE(\infty,\delta)$.
However, a moment of reflection shows that any solution $\tilde z\in\tilde\EE(\infty,\delta)$  
of problem \eqref{tilde-abstract-infty} inherits the additional regularity $\tilde z\in\EE(\infty,\delta)$.

It follows from the implicit function theorem that $\cT$ is $C^1$ in ${\sf y_0}$, 
but only continuous in ${\sf x}_\infty$, unless more regularity for the parameter functions is required.
Nonetheless, $\cT$ is differentiable with respect to ${\sf x}_\infty$ at ${\sf x}_\infty=0.$
%This construction of the stable foliation of the problem. 

\medskip

The properties of $\cT$ can be summarized as follows:
given an equilibrium 
$$z_\infty = {\sf x}_\infty + \psi({\sf x}_\infty) + \phi({\sf x}_\infty + \psi({\sf x}_\infty))\in\cE$$
and an initial value ${\sf y}_0\in X^s_\gamma$, one determines with the help of the
implicit function theorem a value ${\sf x}_0$ in $X^c_0$ and
a solution $z$  of \eqref{full-problem} with initial value $({\sf x}_0,{\sf y}_0)$ 
such that $z(t)$ convergences to $z_\infty$ exponentially fast.
Exponential convergence is obtained by setting up the implicit function theorem
in a space of exponentially decaying functions.
Next, the mapping  
\begin{equation*}
S: B_{X_0^c}(0,r)\times B_{X_\gamma^s}(0,r)\to X_0^c\times X_\gamma^s,\quad
({\sf x}_\infty,{\sf y}_0)\mapsto({\sf x}_0,{\sf y}_0),
\end{equation*}
will be analyzed in more detail.
It is worthwhile to point out that this mapping 
gives rise to the construction of a stable foliation for \eqref{full-problem}
(by invariant, locally stable manifolds) in a neighborhood of the fixed equilibrium $e_*\in\cE$,
see \cite{PSW13}.

To complete the proof, the question which remains is whether $S$ is surjective near $(0,0)$.
Indeed, surjectivity would imply that for any initial value $({\sf x}_0,{\sf y}_0)$ in a sufficiently
small neighborhood of $(0,0)$ in $X^c_0\times X^s_\gamma$ 
there exists $z_\infty\in\cE$ and a unique solution $z$ to problem~\eqref{full-problem} 
which converges to $z_\infty$ at an exponential rate
in the topology of $\cSM$.

To prove surjectivity of $S$, degree theory will be employed. 
For this purpose, define a map $f: B_{X^c_0}(0,r) \times B_{X^s_\gamma}(0,r) \to X_0^c$ 
by means of $ f({\sf x}_\infty,{\sf y}_0)={\sf x}_0({\sf x}_\infty,{\sf y}_0)$. 
As has already been established, this map is continuous,
and it is close to the identity. 
In fact, differentiating the relation
\begin{equation*}
\HH_1(\cT({\sf x_\infty},{\sf y}_0),\cT({\sf x}_\infty,{\sf y}_0))=0
\end{equation*}
with respect to $({\sf x}_\infty,{\sf y}_0)$ at $(0,0)$
% and using \eqref{HH-derivative} 
one obtains
$(D_1 \cT_1(0,0), D_2 \cT_1(0,0))=0.$
Here $\HH_1$ denotes the first line of $\HH$ and $\cT_1$ the first component of $\cT$, respectively.
This implies
$(D_1 \overline{z}(0,0), D_2\overline{z}(0,0))=0.$
From the representation
$$f({\sf x}_\infty,{\sf y}_0) = {\sf x}_\infty - \omega\int_0^\infty P^c\PP\,\overline{z}\,ds,$$
one infers that for every $\ve >0$ there is a constant $\rho>0$ such that
$$ |f({\sf x}_\infty,{\sf y}_0) -{\sf x}_\infty|_{X_0^c} 
\leq \omega\int_0^\infty |P^c\PP\,\overline{z}|_{X_0^c} \,ds
\leq \varepsilon (|{\sf y}_0|_{X^s_\gamma}+|{\sf x}_\infty|_{X_0^c}),
$$
whenever $|({\sf x}_\infty, {\sf y}_0)|\le \rho,$ with $\rho\le r$.
In the following, let $\varepsilon =1/3$ be fixed.
Here ${\sf y}_0$ only serves as a parameter, 
so we are in a finite dimensional setting and may employ the Brouwer degree, 
in particular its homotopy invariance. Define the homotopy $h(\tau,{x},{\sf y}_0) =\tau f({x},{\sf y}_0)+(1-\tau){x}$, and consider the degree
$$ {\sf deg }(h( \tau, \cdot,{\sf y}_0), B_{X_0^c}(0,r),\xi),\quad 
(\xi,{\sf y_0})\in B_{X_0^c}(0,\rho/2)\times B_{X_\gamma^s}(0,\rho/2).$$
For $\tau =0$ it is equal to one, hence  it is equal to one for all $\tau\in[0,1]$, provided there are no solutions of 
$h(\tau, {x},{\sf y}_0)=\xi$ with $|{x}|_{X^c_0}=r$.
To show this, suppose $h(\tau, {x},{\sf y}_0)=\xi,$ i.e., $\xi-{x} = \tau (f({x},{\sf y}_0)-{x})$, and $|{x}|_{X^c_0}=r$. Then by the above estimate
$$ r=|{x}|_{X_0^c}\leq |\xi|_{X_0^c} + |{x}-\xi|_{X^c_0}  
\leq |\xi|_{X_0^c} + \varepsilon (|{\sf y}_0|_{X^s_\gamma}+|{x}|_{X_0^c})  <r,$$
provided $ |\xi|_{X_0^c} < \rho/2$ and  $|{\sf y}_0|_{X^s_\gamma}<\rho/2.$
Hence,  ${\sf deg}(f(\cdot,{\sf y}_0),  B_{X_0^c}(0,r/2),\xi)$ equals one as well, showing that the 
equation $f({x}_\infty,{\sf y}_0)=\xi$ has at least one solution for
each $(\xi,{\sf y_0})\in B_{X_0^c}(0,\rho/2)\times B_{X_\gamma^s}(0,\rho/2)$, i.e., the mapping is surjectiv near zero.
This completes the proof of the theorem.
\end{proof}
\begin{remark}
It should be noted that the proof of surjectivity can be based on the inverse function theorem (in lieu of employing degree theory),
provided the mappings involved are $C^1$ in all variables. This property can be ensured by
asking for one more degree of regularity for the functions $\psi_i$ and $d_i,\upmu_i$ in 
condition \eqref{condition-H}.
\end{remark}

%%%%%%%%%%%%%%%%%%%%%%%%%%%%%%%%%%%%%%%%%%%%%%%%%%%%%%%%%%%%%
\section{Global Existence and Convergence}\label{sect-conv}
%%%%%%%%%%%%%%%%%%%%%%%%%%%%%%%%%%%%%%%%%%%%%%%%%%%%%%%%%%%%%
It has been shown in Section 2 that the negative total entropy is a strict Lyapunov functional 
for \eqref{NS-heat}-\eqref{kinematic}.
Therefore, the {\em $\omega$-limit sets}
\begin{equation*}
\begin{aligned}
\omega(u,\theta,\Gamma):=&\big\{ (u_\infty,\theta_\infty,\Gamma_\infty)\in \cSM: \\
&\ \exists\ t_n\nearrow\infty\ \mbox{s.t.}\ 
(u(t_n),\theta(t_n),\Gamma(t_n))\to (u_\infty,\theta_\infty,\Gamma_\infty)\ \mbox{in}\ \cSM \big\},
\end{aligned}
\end{equation*}
of solutions $(u,\theta,\Gamma)$ in $\cSM$ are contained in the manifold $\cE\subset \cSM$ of equilibria.
There are several obstructions for global existence:
\begin{itemize}
\item {\em regularity}: the norms of either $u(t)$, $\theta(t)$, $\Gamma(t)$  may become unbounded;
\vspace{1mm}
\item {\em geometry}: the topology of the interface may change;\\
    or the interface may touch the boundary of $\Omega$;\\
    or a part of the interface may contract to a point.
\end{itemize}

Let $\Gamma\subset \Omega$ be a hypersurface.
Then $\Gamma$ satisfies the \highlight{\em ball condition} if there is a number $r>0$ such
that for each point $p\in\Gamma$ there are balls $B(x_i,r)\subset \Omega_i$ such that
$\Gamma\cap \bar B(x_i,r)=\{p\}$ for $i=1,2$.
The subset  $\cMH^2(\Omega,r)\subset \cMH^2(\Omega)$
consists, by definition, of all hypersurfaces $\Gamma\in \cMH^2(\Omega)$
that satisfy the ball condition for a fixed radius $r>0$.

Let  $(u,\theta,\Gamma)$ be a solution of \eqref{NS-heat}-\eqref{kinematic}
on its maximal existence interval $[0,t_+)$.
Then $\Gamma(t)$ is said to satisfy the \highlight{\em uniform ball condition} if there exists a number $r>0$ such that
$\Gamma(t)\in\cMH^2(\Omega,r)$ for all $t\in [0,t_+)$.
Note that this condition
bounds the curvature of $\Gamma(t)$, prevents parts of $\Gamma(t)$ to shrink to points, to touch the outer
boundary $\partial \Omega$, and to undergo topological changes.

With this property, combining the local semiflow for problem \eqref{NS-heat}-\eqref{kinematic} with
the corresponding Lyapunov functional (i.e., the negative total entropy), relative compactness of bounded orbits,
and the convergence results from the previous section, one obtains the following global  result.
\goodbreak
%%%%%%%%%%%%%%%%
\begin{theorem} 
\label{qual-heat} Let $p>n+2$ and suppose that condition \eqref{condition-H} holds.  

\medskip
\noindent
Suppose that $(u,\theta,\Gamma)$ is a solution of
\eqref{NS-heat}-\eqref{kinematic} in the state manifold $\cSM$ on its maximal time interval $[0,t_+)$.
Assume there are constants $M,m>0$  such that the  following conditions hold on $[0,t_+)$:
\begin{itemize}
\item[{\bf (i)}]  $|u(t)|_{{W^{2-2/p}_p}},\;|\theta(t)|_{W^{2-2/p}_p},\;|\Gamma(t)|_{W^{3-2/p}_p}\leq M<\infty$, 
\vspace{2mm}
\item[{\bf (ii)}]  $m\leq\theta(t)$;
\vspace{2mm}
\item[{\bf (iii)}] $\Gamma(t)$ satisfies the uniform ball condition.
\end{itemize}
Then $t_+=\infty$, i.e., the solution exists globally, and the solution converges 
in $\cSM$ to an equilibrium $(0,\theta_\infty,\Gamma_\infty)\in\cE$.
\smallskip\\
On the contrary, if $(u(t),\theta(t),\Gamma(t))$ is a global solution in $\cSM$ which converges to an equilibrium $(0,\theta_\infty,\Gamma_\infty)$  in $\cSM$ as $t\to\infty$, then {\rm (i)-(iii)} hold.
\end{theorem}
\begin{proof}
It is well known that each $\Gamma\in\cMH^2(\Omega)$ admits a tubular neighborhood
$U_a:=\{x\in\R^n: {\rm dist}(x,\Gamma)<a\}$ of width $a=a(\Gamma)>0$ such that the signed distance function
$$d_\Gamma:U_a\to \R,\quad |d_\Gamma(x)|:={\rm dist}(x,\Gamma),$$ 
is well-defined and $d_\Gamma\in C^2(U_a,\R)$. 
Here  $d_\Gamma(x)<0$ 
iff $x\in\Omega_1\cap U_a$ by convention.
One then defines a {level function} $\varphi_\Gamma$ by means of
\begin{equation*}
\label{level}
\varphi_\Gamma(x):=
\left\{
\begin{aligned}
& d_\Gamma(x)\chi(3d_\Gamma(x)/a)+ {\rm sgn}\,(d_\Gamma(x))(1-\chi(3d_\Gamma(x)/a)),  &&x\in U_a,\\
& \chi_{\Omega_{\rm ex}}(x)-\chi_{\Omega_{\rm in}}(x),  &&x\notin U_a,
\end{aligned}
\right.
\end{equation*}
where $\Omega_{\rm ex}$ and $\Omega_{\rm in}$ denote the exterior  and interior component of 
$\R^n\setminus U_a$,
respectively,
and $\chi$ is a smooth cut-off function with $\chi(s)=1$ if $|s|<1$ and $\chi(s)=0$ 
if $|s|>2$.
% and $\varphi=1$ in the exterior component of $\RR^n\setminus U_a$,  $\varphi=-1$ in its interior component.
The level function $\varphi_\Gamma$ is then of class $C^2$,
$\varphi_\Gamma(x)=d_\Gamma(x)$ for $x\in U_{a/3}$, and
$\varphi_\Gamma(x)=0$ iff $x\in \Gamma$.

Let $\cMH^2(\Omega,r)$ denote the subset of $\cMH^2(\Omega)$ which consists of all $\Gamma\in \cMH^2(\Omega)$ such that $\Gamma\subset \Omega$ satisfies the ball condition with fixed radius $r>0$. This implies in particular that ${\rm dist}(\Gamma,\partial\Omega)\geq 2r$  and all principal curvatures of $\Gamma\in\cMH^2(\Omega,r)$ are bounded by $1/r$. Furthermore, 
the level functions $\varphi_\Gamma $ are well-defined 
for $\Gamma\in\cMH^2(\Omega,r)$, and form a bounded subset of $C^2(\bar{\Omega})$. The map 
$$\Phi:\cMH^2(\Omega,r)\to C^2(\bar{\Omega}),\quad \Phi(\Gamma)=\varphi_\Gamma,$$ 
is a homeomorphism of the metric space $\cMH^2(\Omega,r)$ onto $\Phi(\cMH^2(\Omega,r))\subset C^2(\bar{\Omega})$, see \cite[Section 2.4.2]{PrSi16}.

Let $s-(n-1)/p>2$. For $\Gamma\in\cMH^2(\Omega,r)$
one defines $\Gamma\in W^s_p(\Omega,r)$ if $\varphi_\Gamma\in W^s_p(\Omega)$. In this case the local charts for $\Gamma$ can be chosen of class $W^s_p$ as well. A subset $A\subset W^s_p(\Omega,r)$ is said to be (relatively) compact, if $\Phi(A)\subset W^s_p(\Omega)$ is (relatively) compact. Finally, 
one defines
${\rm dist}_{W^s_p}(\Gamma_1,\Gamma_2):= |\varphi_{\Gamma_1}-\varphi_{\Gamma_2}|_{W^s_p(\Omega)}$
for $\Gamma_1,\Gamma_2\in \cMH^2(\Omega,r)$.

Suppose that the assumptions (i)-(iii) are valid. 
Then $\Gamma([0,t_+))\subset W^{3-2/p}_p(\Omega,r)$ is bounded, hence relatively compact in
$W^{3-2/p-\ve}_p(\Omega,r)$. Thus $\Gamma([0,t_+))$ can be covered by finitely many balls with centers $\Sigma_k$  such that
%${\rm dist}_{W^{3-2/p-\ve}_p}(\Gamma(t),\Sigma_j)\leq \delta$ for some $j=j(t)$, $t\in[0,t_+)$. 
$${\rm dist}_{W^{3-2/p-\ve}_p}(\Gamma(t),\Sigma_j)\leq \delta \quad\text{for some $j=j(t)$, $t\in[0,t_+)$}.$$ 
Let $J_k=\{t\in[0,t_*):\, j(t)=k\}$. Using for each $k$ a Hanzawa-transformation $\Xi_k$, we see that the pull backs $\{(u(t,\cdot),\theta(t,\cdot))\circ\Xi_k:\, t\in J_k\}$ are bounded in $W^{2-2/p}_p(\Omega\setminus \Sigma_k)^{n+1}$, 
 hence relatively compact in $W^{2-2/p-\ve}_p(\Omega\setminus\Sigma_k)^{n+1}$.
Employing now \cite[Theorem 9.2.1]{PrSi16} one obtains solutions
 $(u^1,\theta^1,\Gamma^1)$ with initial configurations $(u(t),\theta(t),\Gamma(t))$ in the state manifold $\cSM$ on a common time interval, say $(0,a]$, and by uniqueness one has
 $$(u^1(a),\theta^1(a)\Gamma^1(a))=(u(t+a),\theta(u+a),\Gamma(t+a)).$$ 
 Continuous dependence implies then relative compactness of $\{(u(\cdot),\theta(\cdot),\Gamma(\cdot)):\, 0\leq t<t_+\}$ in $\cSM$; 
in particular $t_+=\infty$ and the orbit $(u,\theta,\Gamma)(\R_+)\subset\cSM$ is relatively compact.
The entropy is a strict Lyapunov functional, hence the limit set $\omega(u,\theta,\Gamma)$
of a solution is contained in the set $\cE$ of equilibria.
By compactness, $\omega(u,\theta,\Gamma)\subset \cSM$ is non-empty, 
hence the solution comes close to $\cE$. Finally, one may apply the convergence result 
Theorem \ref{nonlinear-stability} to complete the sufficiency part of the proof. 
Necessity follows by a compactness argument.
\end{proof}
%%%%%%%%%%%%%%%%%%%%%%%%%%%%%%%%%%%%%%%%%%%
\section{The isothermal problem}
%%%%%%%%%%%%%%%%%%%%%%%%%%%%%%%%%%%%%%%%%%%%%%
In this section  the isothermal Navier-Stokes problem with surface tension~\eqref{NS-iso} will be considered.
It turns out that the main results for this problem 
concerning well-posedness and the stability analysis of equilibria
parallel those in Sections 2-4 for problem \eqref{NS-heat}-\eqref{kinematic}.
A decisive difference is caused by the fact that,
as temperature is neglected, the principles of thermodynamics do no longer apply.

The basic local well-posedness result for problem \eqref{NS-iso} reads as follows. 
%Here $\cP_\Gamma=I-\nu_\Gamma\otimes \nu_\Gamma$ denotes the orthogonal projection onto the tangent space of $\Gamma$.
%%%%%%%%%%%%%%%%%%%%%%%%
\begin{theorem}
\label{thm:existence-isothermal}
Let $p>n+2$ and suppose that $\partial\Omega\in C^3$.
Assume the regularity conditions
\begin{equation*}
 u_0\in W^{2-2/p}_p(\Omega\setminus\Gamma_0), \quad \Gamma_0\in W^{3-2/p}_p,
\end{equation*}
and the compatibility conditions
\begin{equation*}
\begin{aligned}
&{\rm div}\, u_0=0 \;\;{\rm in}\;\; \Omega\setminus\Gamma_0,\quad u_0=0\;\;{\rm on}\;\; \partial\Omega, \\
&[\![u_0]\!]=0,\;\;\cP_{\Gamma_0}\,[\![\upmu D(u)\nu_{\Gamma_0}]\!]=0
\;\; {\rm on}\;\; \Gamma_0.
\end{aligned}
\end{equation*}
Then there exists a number $a=a(\Gamma_0,u_0)$ and a unique classical solution $(u,\pi,\Gamma)$
of \eqref{NS-iso} on the time interval $(0,a)$.
Moreover, 
${\mathcal M}=\bigcup_{t\in (0,a)}\{t\}\times \Gamma(t)$
is real analytic.
\end{theorem}
\begin{proof}
The reader is referred to \cite{KPW13}, see also \cite[Sections 9.2 and~9.4]{PrSi16}.
\end{proof}
%%%%%%%%%%%%%%%%%%%
\noindent
The state manifold for \eqref{NS-iso} is defind by
\begin{eqnarray}
\label{phasemanif0}
\cSM:=&&\hspace{-0.5cm}\Big\{(u,\Gamma)\in C(\bar{\Omega})^{n}\times \cMH^2:
   u\in W^{2-2/p}_p(\Omega\setminus\Gamma)^{n},\, \Gamma\in W^{3-2/p}_p,\nonumber\\
 && {\rm div}\, u=0\;\; \mbox{in}\;\; \Omega\setminus\Gamma,\quad u=0\;\; \mbox{on}\;\; \partial\Omega,
 \quad\cP_\Gamma\, [\![\upmu D(u)\nu_\Gamma]\!]=0 \;\;\mbox{on}\;\;\Gamma
 \Big\}.\nonumber
\end{eqnarray}
Applying Theorem~\ref{thm:existence-isothermal} and re-parameterizing the interface repeatedly, one
shows once more that problem~\eqref{NS-iso} generates a semiflow on $\cSM$.
\noindent
As in problem~\eqref{NS-heat}-\eqref{kinematic}
the pressure $\pi$ is determined for each time $t$ from
$(u,\Gamma)$ by means of the weak transmission problem 
%%%%%%%%%%%%%%%%%%%%%%
\begin{equation*}
\begin{aligned}\label{pressure-reconstr}
\left.\left(\varrho^{-1}\nabla{\pi}\,\right|\nabla\phi\right)_{L_2(\Omega)}&=
\left.\left(\varrho^{-1}\upmu\Delta u-(u|\nabla) u\,\right|\nabla\phi\right)_{L_2(\Omega)},
\quad \phi\in H^{1}_{p^\prime}(\Omega),\\
[\![\pi]\!]&= \sigma H_\Gamma +([\![2\upmu D(u)\nu_\Gamma]\!]\,|\nu_\Gamma)
\;\;\mbox{on}\;\; \Gamma.
\end{aligned}
\end{equation*}
%%%%%%%%%%%%%%%%%%%%%%%%
\medskip\\
\noindent
Suppose that the dispersed phase $\Omega_1$ consists of $m$ connected disjoint components,
 $\Omega_1=\bigcup_{k=1}^m \Omega_{1,k}$. Let
$\Gamma_{\!k}:=\partial\Omega_{1,k}$, $\Gamma=\bigcup_{k=1}^m\Gamma_{\!k}$,
and let ${\sf M}_k:=|\Omega_{1,k}|$ denote the volume of $\Omega_{1,k}$. 
Then one shows exactly as in Section~2.3
that problem \eqref{NS-iso} preserves the volume of each individual phase component. 

The \highlight{\em available energy} for problem~\eqref{NS-iso} is defined by
$$
\Phi_0(u,\Gamma):=\frac{1}{2}\int_{\Omega\setminus\Gamma}\varrho |u|^2\,dx +\sigma |\Gamma|.
$$
For the time derivative of $\Phi_0$ one obtains
\begin{align*}
\frac{d}{dt}\Phi_0 &= \int_\Omega \varrho(\partial_t u|u) \,dx
-\int_\Gamma \{[\![\frac{\varrho}{2}|u|^2 ]\!]+\sigma H_\Gamma\}V_\Gamma\,d\Gamma\\
&= -\int_\Omega  \{\varrho((u|\nabla)u|u) - ({\rm div}\,T | u)\}\,dx 
   -\int_\Gamma\big \{[\![\frac{\varrho}{2}|u|^2 ]\!] +\sigma H_\Gamma\big\}V_\Gamma\, d\Gamma\\
&= - 2\int_{\Omega} \upmu\, |D(u)|_2^2\,dx 
- \int_\Gamma\big\{[\![(Tu|\nu_\Gamma) ]\! ] +\sigma H_\Gamma V_\Gamma\big\}\,d\Gamma \\
&=- 2\int_{\Omega}\upmu\, |D(u)|_2^2\,dx, 
\end{align*}
showing that the available energy is decreasing.
Therefore, $\Phi_0$ constitutes a Lyapunov functional for \eqref{NS-iso}.
As in Section 2.5 one shows that $\Phi_0$ is a strict Lyapunov functional.
The same arguments as in Section 2 also imply that
the equilibria of \eqref{NS-iso}
consist of zero velocities, constant pressures in the phase components, and 
that the dispersed phase consists of a collection of non-intersecting balls in $\Omega$.
Consequently, the set $\cE$ of non-degenerate equilibria for \eqref{NS-iso} is given by
$$\cE=\{(0,\Sigma): \Sigma\in \cS\},$$
where $\cS$ is defined in~\eqref{def-S}. 
$\cE$ defines a real analytic manifold of dimension $m(n+1)$.

In analogy to Section 2.7 one shows that the critical points of the energy functional $\Phi_0$
under the constraints of ${\sf M}_k={\sf M}_{0,k}$ constant
correspond exactly to the equilibria of \eqref{NS-iso}, 
and  that all critical points are local minima of the energy functional
with the given constraints.
\goodbreak
%%%%%%%%%%%%%%%%%%%%%%
\begin{theorem}
The following assertions hold for problem \eqref{NS-iso}.
\begin{itemize}
\vspace{1mm} \item
[{\bf (a)}] The phase volumes $|\Omega_{1,k}|$ are preserved. 
\vspace{1mm} \item
[{\bf (b)}] The energy functional $\Phi_0$ is a strict Lyapunov functional.
\vspace{1mm}\item
[{\bf (c)}] The {\em non-degenerate equilibria} are zero velocities, constant pressures in the
components of the phases,  and the interface is a finite union of non-intersecting spheres which 
do not touch the outer boundary $\partial\Omega$. 
\vspace{1mm} \item
[{\bf (d)}] The set $\cE$ of non-degenerate equilibria forms a real analytic manifold
of dimension $m(n+1)$, where $m$ denotes the number of connected components of $\Omega_1$.
\vspace{1mm} \item
[{\bf (e)}] The critical points of the energy functional
for prescribed phase volumes are precisely the equilibria of the system. 
\vspace{1mm} \item
[{\bf (f)}] All critical points  of the energy functional
for prescribed phase volumes are local minima.
\end{itemize}
\end{theorem}
%%%%%%%%%%%%%%%%%%%%%%%%%%%%%%%%%%%%%%%%%%%%%%%%%%
\noindent
This result was first established in \cite[Proposition~5.2]{KPW13}, 
see also  \cite[Theorem 3.1]{BoPr10b}.
It should be observed that the assertions in Remark 2.3 do also apply to the isothermal case.

\medskip
\noindent
Suppose  $(0,\Sigma)\in\cE$, with 
$\Sigma=\bigcup_{k=1}^m\Sigma_k$ and $\Sigma_k=\partial B(x_k,R_k)$
is a fixed equilibrium for problem~\eqref{NS-iso}.
In analogy to Section 3, and using the notation introduced there, 
one associates with system \eqref{NS-iso} the following linear problem
\begin{equation}
\begin{aligned}\label{lin-uh-iso}
\varrho\partial_t u-\upmu \Delta u +\nabla\pi &=\varrho f_u && \mbox{in} && \Omega\setminus{\Sigma},\\
{\rm div}\, u &=g_d && \mbox{in} && \Omega\setminus{\Sigma},\\
u&=0\quad && \mbox{on} && \partial\Omega,\\
[\![u]\!] &= 0      && \mbox{on} && {\Sigma},\\
 -[\![T\nu_\Sigma]\!] +\sigma (\cA_\Sigma h)\nu_\Sigma &=g_u && \mbox{on} && {\Sigma},\\
\partial_t h- (u|\nu_\Sigma ) &= f_h  &&\mbox{on} && \Sigma,\\
u(0)&=u_0 &&\mbox{in} && \Omega, \\
h(0)&=h_0 &&\mbox{on} && \Sigma,\\
\end{aligned}
\end{equation}
and the linear operator $L$,
$$ L(u,h):= \big(-\varrho^{-1}(\upmu \Delta u -\nabla\pi), -(u|\nu_\Sigma)\big),$$
defined on 
$X_0=L_{p,\sigma}(\Omega)\times W^{3-1/p}_p(\Sigma)$ with domain
\begin{equation*}
\begin{aligned}
X_1= {\sf D}(L)&=\{(u,h)\in H^2_p(\Omega\setminus\Sigma)^{n}\times W^{3-1/p}_p(\Sigma): 
{\rm div}\, u=0\; \mbox{ in }\; \Omega\setminus\Sigma,\\ 
&\hspace{0.7cm} u=0\;\mbox{ on }\;\partial\Omega,\;\;
[\![u]\!],\;\cP_\Sigma[\![\upmu  D(u)\nu_\Sigma]\!]=0\;\:\mbox{on}\:\; \Sigma
\}.
\end{aligned}
\end{equation*}
Here $\pi$ is again determined as the solution of the weak transmission problem
\begin{equation*}
\begin{aligned}
(\varrho^{-1}\nabla\pi|\nabla\phi)_{L_2(\Omega)} 
&=(\varrho^{-1}\upmu\Delta u|\nabla \phi)_{L_2(\Omega)},\quad \phi\in {H}^1_{p^\prime}(\Omega),\\
[\![\pi]\!] &=-\sigma \cA_\Sigma h + ([\![2\upmu D(u)\nu_\Sigma]\!]\,|\nu_\Sigma)\quad \mbox{ on } \Sigma.
\end{aligned}
\end{equation*}
The principal variable in \eqref{lin-uh-iso} is $z=(u,h)$, the dynamic inhomogeneities
are $f=(f_u,f_h)$, and the static ones are $g=(g_d, g_u)$. 
The eigenvalue problem associated with $L$ becomes
\begin{equation}
\begin{aligned}\label{ev-lin-iso}
\varrho\lambda  u-\upmu \Delta u +\nabla\pi &= 0 && \mbox{in} && \Omega\setminus{\Sigma},\\
{\rm div}\, u &= 0 && \mbox{in} && \Omega\setminus{\Sigma},\\
u&=0  && \mbox{on} && \partial\Omega,\\
[\![u]\!] &= 0      && \mbox{on} && {\Sigma},\\
 -[\![T\nu_\Sigma]\!] +\sigma (\cA_\Sigma h)\nu_\Sigma &= 0 && \mbox{on} && {\Sigma},\\
\lambda  h- (u|\nu_\Sigma) &= 0 &&  \mbox{on} && {\Sigma}.\\
\end{aligned}
\end{equation}
%%%%%%%
\goodbreak
%%%%%%%%%%%%%%%
\begin{theorem}
\label{linstab-iso} 
Let $e_*\in\cE$ be an equilibrium. 
Then the operator $L$ has the following properties.
\begin{itemize}
\vspace{1mm}
\item[{\bf (a)}]
$-L$ generates a compact, analytic $C_0$-semigroup in $X_0$ which has 
the property of maximal $L_p$-regularity. 
\vspace{1mm}
\item[{\bf (b)}]
The spectrum of $L$ consists of countably many eigenvalues of finite algebraic multiplicity. 
\vspace{1mm}
\item[{\bf(c)}] $-L$ has no eigenvalues $\lambda\neq0$ with nonnegative real part. 
\vspace{1mm}
\item[{\bf (d)}] $\lambda=0$ is a semi-simple eigenvalue of $L$ of multiplicity $m(n+1)$. 
\vspace{1mm}
\item[{\bf (e)}] The kernel ${\sf N}(L)$ of $L$ is isomorphic to the tangent space $T_{e_*}\cE$ of the manifold of equilibria $\cE$ at $e_*$.
\end{itemize}
Hence, each equilibrium $e_*\in\cE$ is normally stable.
\end{theorem}
%%%%%%%%%%%%%
\begin{proof}
The proof proceeds in the same way as the corresponding proof of Theorem~\ref{linstab-heat}, 
with the only difference that here all  quantities and assertions relating to the temperature are dismissed.
One verifies, for instance, that the kernel of $L$ is spanned by the functions
$e_{jk}=(0,Y_k^j)$ 
with $Y_k^j$ the spherical harmonics  of degree one for the spheres $\Sigma_k$, 
$j=1,\cdots,n$, $k=1,\cdots,m$, 
and $e_{0,k}=(0,Y_k^0)$, where $Y_k^0$ equals one on $\Sigma_k$ and zero elsewhere.
Hence the dimension of the null space ${\sf N}(L)$ is $m(n+1)$. 
\end{proof}
The main theorem of this chapter concerning the stability of equilibria reads as follows.
%%%%%%%%%%%%%%%%%%%%%%%%
\begin{theorem} \label{nonlinear-stability-iso}
Let $p>n+2$.
Then every equilibrium $e_*=(0,\Sigma)\in\cE$
is nonlinearly stable in the state manifold $\cSM$. 
Any solution with initial value close to $e_*$ in $\cSM$ exists globally
and converges in $\cSM$ to a possibly different stable equilibrium $e_\infty\in\cE$ at an exponential rate.
\end{theorem}
%%%%%%%%%%%%%%%
\begin{proof}
The proof proceeds in the same way as the corresponding proof for Theorem~\ref{nonlinear-stability},
with the obvious modification that all quantities and assertions relating to the temperature
are to be disregarded.
\end{proof}
Analogous assertions as in Theorem~\ref{qual-heat} hold for solutions
$(u(t),\Gamma(t))$ of the isothermal problem \eqref{NS-iso},
again with the obvious modification that the temperature variable is dropped,
see also \cite[Theorem 7.1]{KPW13}.

%%%%%%%%%%%%%%%%%%%%%%%%%%%%%%%%%%%%%%%%%%%%%%%%%%%%%%
\section{The two-phase Stokes flow with surface tension}
%%%%%%%%%%%%%%%%%%%%%%%%%%%%%%%%%%%%%
In this section the two-phase
quasi-stationary Stokes problem with surface tension~\eqref{Stokes-flow} will be considered.
This problem is considerably easier to analyze than  
problem~\eqref{NS-heat}-\eqref{kinematic},
or problem \eqref{NS-iso}. In fact, it turns out that in this case, the only system variable is the unknown 
hypersurface $\Gamma$. In order to obtain this reduction,   
the two-phase Stokes problem
\begin{equation}
\label{Stokes-stationary}
\begin{aligned}
-\upmu\Delta u +\nabla \pi &=0 &&\mbox{in} &&\Omega\setminus\Gamma,\\
{\rm div}\, u &=0 &&\mbox{in} && \Omega\setminus\Gamma,\\
u &=0 &&\mbox{on} && \partial\Omega, \\
[\![u]\!]&=0 &&\mbox{on} && \Gamma, \\
- [\![T\nu_\Gamma]\!]&=g\nu_\Gamma &&\mbox{on} && \Gamma
\end{aligned}
\end{equation}
will play an important role.
One shows  that \eqref{Stokes-stationary} admits for each $g\in W^{1-1/p}_p(\Gamma)$ 
a unique solution $(u,\pi)$ (up to constants in the pressure) with regularity
$$(u,\pi)\in H^2_p(\Omega\setminus\Gamma)\times \dot H^1_p(\Omega\setminus\Gamma),$$
where $1<p<\infty$, and $\dot H^1_p$ denotes the homogeneous Sobolev space of order one.

For a given function $g\in W^{1-1/p}_p(\Gamma)$ let $(u,\pi)$ be the solution of \eqref{Stokes-stationary}.
Then  the Neumann-to-Dirichlet operator $N_\Gamma:W^{1-1/p}_p(\Gamma)\to W^{2-1/p}_p(\Gamma) $ is defined by
$$ N_\Gamma g := (u|\nu_\Gamma).$$
The following results hold for $N_\Gamma$.
\goodbreak
%%%%%%%%%%%%%%%%%%%%%%%%
\begin{proposition}
\label{ND-Stokes-properties}
Suppose $\partial\Omega\in C^3$ and  $\Gamma\in\cMH^2(\Omega)$ consists of $m$ components, 
$\Gamma=\bigcup_{k=1}^m \Gamma_k$. Then the operator $N_\Gamma$ has the following properties.
\begin{itemize}
\item
[{\bf(a)}] 
$(N_\Gamma g|h)_{L_2(\Gamma)}=(g|N_\Gamma h)_{L_2(\Gamma)}$, \;\;$g,h\in W^{1/2}_2(\Gamma)$.
\vspace{2mm}\item
[{\bf(b)}] 
$(N_\Gamma g|g)_{L_2(\Gamma)}=2\int_\Omega \upmu |D(u)|_2^2\,dx$, where $(u,\pi)$ is the solution of 
       \eqref{Stokes-stationary}
\vspace{2mm}\item
[{\bf(c)}] 
Let ${\sf e}_k$ be the function which is one on $\Gamma_k$ and zero elsewhere.
Then 
$$\mbox{$(N_\Gamma g|{\sf e}_k)_{L_2(\Gamma)}=0$ for each $g\in W^{1/2}_2(\Gamma)$ and each $1\le k \le m$}.$$
In particular, $N_\Gamma {\sf e}_k=0$ for each $k$, and $N_\Gamma g$ has mean value zero 
for each function $g\in W^{1/2}_2(\Gamma)$.   
\vspace{2mm}\item
 [{\bf(d)}] ${\sf N}(N_\Gamma) = {\rm span}\{{\sf e}_1,\cdots,{\sf e}_m\}.$ 
\end{itemize}
 \end{proposition}
\begin{proof}
Let $g,h\in W^{1/2}_2(\Sigma)$ be given, and let
$(u,\pi)$ denote the solution of 
\eqref{Stokes-stationary} corresponding to $g$, and $(v,q)$ the solution corresponding to $h$, respectively.
Then one obtains 
\begin{equation*}
\begin{aligned}
(N_\Gamma g|h)_{L_2(\Gamma)}
&= \int_\Gamma (u|h\nu_\Gamma)\,d\Gamma
= -\int_\Gamma [\![(u|T(v,q)\nu_\Gamma) ]\!]\,d\Gamma \\
&= \int_\Omega {\rm div}\,(T(v,q)u)\,dx
= 2\int_\Omega \upmu\, D(v)\!:\!D(u)\,dx
\end{aligned}
\end{equation*}
and the assertions in (a)-(b) follow at once.
Here, $D(u)\!:\!D(v)={\rm trace}\, (D(u)D(v))$ denotes the Frobenius inner product of the (symmetric) matrices $D(u)$ and $D(v)$.

\medskip
\noindent
(c) Let $g\in W^{1/2}_2(\Sigma)$ be given, and let $(u,\pi)$ be the solution of~\eqref{Stokes-stationary}.
By the divergence theorem
\begin{equation*}
(N_\Gamma g|{\sf e}_k)_{L_2(\Sigma)}
=\int_{\Gamma_k} (u|\nu_\Gamma)\,d\Gamma_k = \int_{\Omega_{1,k}}{\rm div}\, u\,dx=0,
\end{equation*}
with $\Omega_{1,k}$ the region enclosed by $\Sigma_k$.
Therefore, by density of $W^{1/2}_2(\Gamma)$ in $L_2(\Gamma)$,
 $(g|N_\Gamma{\sf e}_k)_{L_2(\Sigma)}=0$ for all $g\in L_2(\Sigma)$,
and hence $N_\Gamma {\sf e}_k=0$.

\medskip
\noindent
(d)
It remains to show that ${\sf N}(N_\Gamma)\subset {\rm span}\{{\sf e_1},\cdots,{\sf e_m}\}$.
Suppose  $g\in {\sf N}(N_\Gamma)$. 
It then follows from part (b) that $D(u)=0$.
Korn's inequality and the no-slip boundary condition readily imply that $u=0$ on $\Omega$,
and hence $\pi$ is constant on connected components of $\Omega$.
This shows that $g|_{\Gamma_k}=[\![\pi]\!]|_{\Gamma_k}$ is constant on each boundary component $\Gamma_k$,
and hence $g\in{\rm span}\{{\sf e_1},\cdots,{\sf e_m}\}$. 
\end{proof}
%%%%%%%%%%%%%%%
By means of the Neumann-to-Dirichlet operator $N_\Gamma$, problem~\eqref{Stokes-flow} can be reformulated as 
a geometric evolution equation
\begin{equation}
\label{geometric}
V_\Gamma = \sigma N_\Gamma H_\Gamma,\quad \Gamma(0)=\Gamma_0.
\end{equation}
In order to study problem \eqref{geometric}
one may parameterize $\Gamma$ over an analytic reference manifold $\Sigma$ which is $C^2$ close to $\Gamma_0$.
%see for instance \cite{PrSi13}, or \cite[Section~2.3]{PrSi16}.
Problem \eqref{geometric} can then be cast as a
quasilinear evolution equation
\begin{equation}
\label{quasilinear}
\partial_t h +A(h)h=F(h),\quad h(0)=h_0,
\end{equation}
where $h(t)$ denotes the height function which parameterizes $\Gamma(t)$, that is,
$$\Gamma(t)=\{q+h(t,q)\nu_\Sigma(q): q\in \Sigma,\; t\ge 0\}.$$
The resulting problem \eqref{quasilinear} is amenable to the theory of maximal $L_p$-regularity 
for quasilinear parabolic evolution equations, 
see for instance \cite[Chapter 5]{PrSi16} for a comprehensive account of this theory.
The following basic well-posedness result holds true.
%%%%%%%%%%%%%%%%%%%
\begin{theorem}
\label{Stokes-existence}
Suppose $p>n+2$.
Then for each $\Gamma_0\in W^{3-2/p}_p$ there is a number $a=a(\Gamma_0)$ 
and a unique classical solution $\Gamma=\{\Gamma(t):t\in (0,a)\}$ for \eqref{geometric}.
Moreover,
${\mathcal M}=\bigcup_{t\in (0,a)}\{t\}\times \Gamma(t)$ is real analytic.
\end{theorem}
\begin{proof}
The reader is referred to \cite[Section 12.5]{PrSi16}
\end{proof}
%%%%%%%%%%%%%%%%%%%%

\medskip
\noindent
Suppose, as in the previous sections, 
that the dispersed phase $\Omega_1$ consists of $m$ disjoint connected components,
 $\Omega_1=\bigcup_{k=1}^m \Omega_{1,k}$. Let
$\Gamma_{\!k}:=\partial\Omega_{1,k}$, $\Gamma=\bigcup_{k=1}^m\Gamma_{\!k}$,
and let ${\sf M}_k:=|\Omega_{1,k}|$
denote the volume of $\Omega_{1,k}$. 
Then one shows, as in Section~2.3,
that problem problem~\eqref{geometric}, or equivalently problem~\eqref{Stokes-flow},
preserves the volume of each individual phase component.
Indeed, by Proposition~\ref{ND-Stokes-properties}(c)
\begin{equation*}
\frac{d}{dt}|\Omega_{1,k}(t)|
=\int_{\Gamma_{\!k}}V_\Gamma\,d\Gamma_{\!k}
=\sigma \int_{\Gamma_{\!k}}N_\Gamma H_\Gamma\,d\Gamma_{\!k}
= \sigma (N_\Gamma H_\Gamma |{\sf e}_k)_{L_2(\Sigma)}=0.
\end{equation*}
The time derivative of the surface area $|\Gamma(t)|$ is given by
\begin{equation*}
\begin{aligned}
\frac{d}{dt}|\Gamma(t)| 
& = -\int_\Gamma V_\Gamma H_\Gamma\,d\Gamma
  = - \sigma \int_\Gamma  (N_\Gamma H_\Gamma) H_\Gamma\,d\Gamma\\
&  = - \sigma (N_\Gamma H_\Gamma | H_\Gamma)_{L_2(\Gamma)}
 =-2\sigma \int_\Omega \upmu |D(u)|_2^2\,dx,
\end{aligned}
\end{equation*}
where $(u,\pi)$ is the solution of~\eqref{Stokes-stationary} with $g=H_\Gamma$. 
This shows that surface area is decreasing.
Therefore, $\Phi_0(\Gamma)=|\Gamma|$ constitutes a Lyapunov functional for \eqref{geometric}.
As in Section 2.5 one shows that $\Phi_0$ is a strict Lyapunov functional.
An analogous argument as in Section 2 also implies that
the equilibria of \eqref{Stokes-flow}
consist of zero velocities, constant pressures in the phase components, and 
that the dispersed phase consists of a collection of non-intersecting balls in $\Omega$.
Consequently, the set $\cE$ of non-degenerate equilibria for \eqref{geometric} is given by
$$\cE=\{\Sigma : \Sigma \in \cS\},$$
where $\cS$ is defined in~\eqref{def-S}. 
$\cE$ gives rise to a real analytic manifold of dimension $m(n+1)$.

In analogy to Section 2.7 one also shows that the critical points of the area functional $\Phi_0$
under the constraints of ${\sf M}_k={\sf M}_{0,k}$ constant
correspond to the equilibria of \eqref{geometric}, 
and  that all critical points are local minima of the area functional $\Phi_0$ 
under the given constraints.
\goodbreak
%%%%%%%%%%%%%%%%%%%%%%
\begin{theorem}
The following assertions hold for problem \eqref{geometric}.
\begin{itemize}
\vspace{1mm} \item
[{\bf(a)}] The phase volumes $|\Omega_{1,k}|$ are preserved. 
\vspace{1mm} \item
[{\bf (b)}] The area functional $\Phi_0$ is a strict Lyapunov functional.
\vspace{1mm}\item
[{\bf (c)}] Each non-degenerate equilibrium 
consists of a finite union of non intersecting spheres which 
do not touch the outer boundary $\partial\Omega$. 
\vspace{1mm} \item
[{\bf (d)}] The set $\cE$ of non-degenerate equilibria forms a real analytic manifold
of dimension $m(n+1)$, where $m$ denotes the number of connected components of $\Omega_1$.
\vspace{1mm} \item
[{\bf (e)}] The critical points of the area functional
for prescribed phase volumes are precisely the equilibria of the system. 
\vspace{1mm} \item
[{\bf (f)}] All critical points  of the area functional
for prescribed phase volumes are local minima.
\end{itemize}
\end{theorem}
%%%%%%%%%%%%%%%%%%%%%%%%%%%%%%%%%%%%%%%%%%%%%%%%%%
In oder to analyze the stability properties of equilibria for the geometric evolution equation \eqref{geometric}
one may proceed as follows.
Suppose $\Sigma=\bigcup_{k=1}^m{\Sigma_k}\in\cE $ is an equilibrium for \eqref{geometric}.
Choosing $\Sigma$ as a reference manifold one shows that problem \eqref{geometric},
or for that matter also problem \eqref{quasilinear}, can be written as
\begin{equation}
\label{geometric-linearized}
\partial_t h +\sigma N_\Sigma\cA_\Sigma h=G_\Sigma(h),\quad h(0)=h_0,
\end{equation}
where $\cA_\Sigma$ has the same meaning as in Section~4.
The nonlinear function $G_\Sigma$ satisfies $(G_\Sigma(0),G^\prime_\Sigma(0))=0$.

\medskip
\noindent
Let $X_0:=W^{2-1/p}_p(\Sigma)$, $X_1:=W^{3-1/p}_p(\Sigma)$, and set
\begin{equation}
\label{L-Stokes-flow}
L:{\sf D}(L)=X_1\subset X_0\to X_0, \quad L:=\sigma N_\Sigma\cA_\Sigma.
\end{equation}
\goodbreak
%%%%%%%%%%%%%%%%%%%%%%%%%%%%%%
\begin{theorem}
\label{linstab-Stokes-flow} 
The operator $L$ has the following properties.
\begin{itemize}
\vspace{1mm}
\item[{\bf (a)}]
$-L$ generates a compact, analytic $C_0$-semigroup in $X_0$ which has 
the property of maximal $L_p$-regularity. 
\vspace{1mm}
\item[{\bf (b)}]
The spectrum of $L$ consists of countably many real eigenvalues of finite algebraic multiplicity. 
The spectrum is independent of $p$.
\vspace{1mm}
\item[{\bf (c)}] $-L$ has no positive eigenvalues.
\vspace{1mm}
\item[{\bf (d)}] $\lambda=0$ is a semi-simple eigenvalue of $L$ of multiplicity $m(n+1)$. 
\vspace{1mm}
\item[{\bf (e)}] The kernel ${\sf N}(L)$ of $L$ is isomorphic to the tangent space $T_{\Sigma}\cE$.
\end{itemize}
Hence, the equilibrium $\Sigma\in\cE$ is normally stable.
\end{theorem}
%%%%%%%%%%%%%
\begin{proof}
The assertions in (a)-(b) follow from standard arguments.

\medskip\noindent
(c)
Suppose that $\lambda\in{\mathbb C}$, $\lambda\neq 0,$ is an eigenvalue for $-L$, that is,
\begin{equation}
\label{EV-L-Stokes}
 \lambda h +\sigma N_\Sigma\cA_\Sigma h=0
\end{equation}
for some nontrivial function $h\in W^{5/2}_2(\Sigma)$.
Taking the inner product of \eqref{EV-L-Stokes} with $\cA_\Sigma h$ in $L_2(\Sigma)$ yields
\begin{equation}
\label{NA}
 \lambda(h|\cA_\Sigma h)_{L_2(\Sigma)} + \sigma (N_\Sigma\cA_\Sigma h |\cA_\Sigma h)_{L_2(\Sigma)} =0.
\end{equation}
As $N_\Sigma$ and $\cA_\Sigma$ are symmetric, 
this identity implies that $\lambda$ must be real, 
hence the spectrum of $L$ is real.

Suppose that $\lambda>0$.
By Proposition~\ref{ND-Stokes-properties}(c), $(N_\Sigma \cA_\Sigma h|{\sf e}_k)_{L_2(\Sigma)}=0$ 
and consequently, $(h|{\sf e}_k)_{L_2(\Sigma)}=0$ as well, which
implies $(h|\cA_\Sigma h)_{L_2(\Sigma)}\ge 0$.
As $N_\Sigma$ is positive semi-definite on $L_2(\Sigma)$,
see Proposition~\ref{ND-Stokes-properties}(b), one concludes that $(h|\cA_\Sigma h)_{L_2(\Sigma)}=0$.
This yields $\cA_\Sigma h=0$, and then $h=0$ by \eqref{EV-L-Stokes} as $\lambda>0$.

\medskip
\noindent
(d) Suppose $h\in {\sf N}(L)$. Then 
$\cA_\Sigma h=\sum_{k=1}^m a_k{\sf e}_k$ by Proposition~\ref{ND-Stokes-properties}(d).
This implies 
$$h=h_0 -\sum_{k=1}^m (a_kR_k^2/(n-1))\; {\sf e}_k,$$
with $h_0\in {\sf N}(\cA_\Sigma)$, where $R_k$ denotes the radius of the sphere $\Sigma_k$.
As ${\sf N}(\cA_\Sigma)$ is spanned by the spherical harmonics $Y^j_k$ on $\Sigma_k$, we
see that ${\rm dim}\,{\sf N}(L)=m(n+1)$.

Next it will be shown that the eigenvalue $0$ is semi-simple.
Suppose $L^2 h=0.$ Then
$$ N_\Sigma\cA_\Sigma h = h_0 + \sum_{k=1}^m a_k {\sf e}_k,
\quad \mbox{ for some } h_0\in {\sf N}(\cA_\Sigma) \mbox{ and } a_k\in \C.$$
Multiplying this relation with ${\sf e}_j$ in $L_2(\Sigma)$ one obtains $a_k=0$ for all $k$.
%thanks to Proposition~\ref{ND-Stokes-properties}.
Taking the $L_2(\Sigma)$ inner product of the relation $N_\Sigma\cA_\Sigma h = h_0$
with $\cA_\Sigma h$ yields 
$$(N_\Sigma \cA_\Sigma h|\cA_\Sigma h)_{L_2(\Sigma)}=(h_0|\cA_\Sigma h)_{L_2(\Sigma)}=0$$
as $\cA_\Sigma$ is symmetric and $h_0\in {\sf N}(\cA_\Sigma)$.
Therefore, $N_\Sigma\cA_\Sigma h=0$, that is, $h\in{\sf N}(L)$.

\medskip
\noindent
(e)
The assertion follows as ${\sf N}(L)$ and $T_{\Sigma}\cE$ are of the same dimension. 
\end{proof}
\goodbreak

%%%%%%%%%%%%%%%%%%%%
\begin{theorem}
\label{geomeveq-stability} 
Let $p>n+2$
and suppose that $\Sigma$ is a (nondegenerate) equilibrium of \eqref{geometric}.

\noindent
% Then $h=0$ is stable for problem~\eqref{geometric-linearized}  in $W^{3-2/p}_p(\Sigma)$. 
Then any solution of \eqref{geometric-linearized} starting close to $0$ in $W^{3-2/p}_p(\Sigma)$ 
 exists globally and converges to an equilibrium $h_\infty$
 in $W^{3-2/p}_p(\Sigma)$ at an exponential rate.
Here, $h_\infty$ corresponds to some $\Gamma_\infty\in\cE$.
\end{theorem}
%%%%%%%%%%%%%%%%%%%%%%%%%%%%%%%
\begin{proof}
The proof of this result is based on the generalized principle of linearized stability
for quasilinear parabolic equations 
introduced in ~\cite{PSZ09b}, see also Chapter 5 in the monograph~\cite{PrSi16}.
\end{proof}
%%%%%%%%%%%%%%%%%%%%%%%
The state manifold for \eqref{geometric} is defined by means of
$\cSM=\cMH^2(\Omega)$.
The main result of this section reads as follows.
%%%%%%%%%%%%%%%%%%%%%%%%%%
\begin{theorem}
\label{geomatric-global}
Let $p>n+2$. 
Suppose that $\Gamma(t)$ is a solution of \eqref{geometric},
defined on its  maximal existence interval $[0,t_+)$.  
Assume there is a constant $M>0$ such that the following conditions hold on $[0,t_+)$:
\begin{enumerate}
    \vspace{1mm}
    \item[{\bf (i)}] $|\Gamma(t)|_{W^{3-2/p}_p}\le M<\infty$;
    \vspace{1mm}
    \item[{\bf (ii)}] $\Gamma(t)$ satisfies a uniform ball condition.
\end{enumerate}
Then $t_+ =\infty$, i.e., the solution exists globally, and $\Gamma(t)$ converges in $\cSM$ 
to an equilibrium $\Gamma_\infty\in\cE$ at an exponential rate.
The converse is also true: if a global solution converges in $\cSM$ to an equilibrium, 
then {\rm (i)} and {\rm (ii)} are valid.
\end{theorem}
\begin{proof}
The proof is similar to that of Theorem~\ref{qual-heat}, see also \cite[Section 12.5]{PrSi16}.
\end{proof}
%%%%%%%%%%%%%%%
\section{Conclusions}
In this chapter,
the equilibrium states for the 
two-phase Navier-Stokes problem with heat-advection and surface tension \eqref{NS-heat}-\eqref{kinematic},
the two-phase isothermal Navier-Stokes problem with surface tension~\eqref{NS-iso},
and the two-phase Stokes flow with surface tension~\eqref{Stokes-flow}
are characterized.
It is shown that every equilibrium is normally stable, and that
every solution that starts close to an equilibrium exists globally
and converges to a (possibly different) equilibrium at an exponential rate.
Moreover, it is shown that
 the negative total entropy for \eqref{NS-heat}--\eqref{kinematic},
 the available energy for   \eqref{NS-iso},
 and the surface area for \eqref{Stokes-flow} 
constitute strict Lyapunov functionals. 
This implies that solutions which do not develop singularities converge
to an equilibrium in the topology of the state manifold $\cSM$.
%%%%%%%%%%%%%%%%%%%%%%
%\section{Cross-References}
%\begin{itemize}
%\item[] Stokes approximation: Chapter 1.1
%\item[] Hanzawa transformation: Chapter 7.1
%\item[] Fluid flows with phase transitions: Chapter 7.1
%\item[] Local well-posedness: Chapter 7.2.
%\end{itemize}
%\goodbreak
%%%%%%%%%%%%%%%%%%%%%%%%%%%%
%\bibliography{fb,fb-gieri}
%\bibliographystyle{abbrv}
%\end{document}
%%%%%%%%%%%%%%%%%%%%%%%%%%%%%%%%%%%%%%%%%%%%%%%%%%%
\def\cprime{$'$} \def\polhk#1{\setbox0=\hbox{#1}{\ooalign{\hidewidth
  \lower1.5ex\hbox{`}\hidewidth\crcr\unhbox0}}}

%%%%%%%%%%%%%%%%%%%%%%%%%%%%%%%%%

\end{document}